\documentclass[11pt]{article}

\usepackage[hidelinks]{hyperref}
\hypersetup{
  colorlinks   = true, 
  urlcolor     = blue, 
  linkcolor    = blue, 
  citecolor   = red 
}
\usepackage{amsmath,amsthm,amssymb}
\usepackage{graphicx}
\usepackage{bbm}

\newcommand{\T}{\mathbb{T}}

\usepackage{xcolor}
\usepackage[margin=1in]{geometry} 
\usepackage{enumerate,mathtools,mathrsfs,bm,graphicx}
\usepackage[numbers, super]{natbib}
\usepackage[hidelinks]{hyperref}
\newcommand{\N}{\mathbb{N}}									

\newcommand{\R}{\mathbb{R}}

\newcommand{\vertiii}[1]{{\left\vert\kern-0.25ex\left\vert\kern-0.25ex\left\vert #1 
    \right\vert\kern-0.25ex\right\vert\kern-0.25ex\right\vert}}
    
\newcommand{\inner}[2]{\left\langle #1, #2 \right\rangle}
\newcommand{\norm}[1]{\left\Vert #1 \right\Vert}
\newcommand{\abs}[1]{\left\vert #1 \right\vert}

\newtheorem{theorem}{Theorem}[section]

\newtheorem{lemma}[theorem]{Lemma}
\newtheorem{proposition}[theorem]{Proposition}

\newtheorem{definition}[theorem]{Definition}

\begin{document}
	\title{Anisotropic Inviscid Limit for the Navier-Stokes Equations with Transport Noise Between Two Plates}
	\author{Daniel Goodair\footnote{\'{E}cole Polytechnique F\'{e}d\'{e}rale de Lausanne, Lausanne, Switzerland, daniel.goodair@epfl.ch}}
	\date{\today} 
	\maketitle
\setcitestyle{numbers}	
\thispagestyle{empty}
\begin{abstract}
We investigate an anisotropic vanishing viscosity limit of the 3D stochastic Navier-Stokes equations posed between two horizontal plates, with Dirichlet no-slip boundary condition. The turbulent viscosity is split into horizontal and vertical directions, each of which approaches zero at a different rate. The underlying Cylindrical Brownian Motion driving our transport-stretching noise is decomposed into horizontal and vertical components, which are scaled by the square root of the respective directional viscosities. We prove that if the ratio of the vertical to horizontal viscosities approaches zero, then there exists a sequence of weak martingale solutions convergent to the strong solution of the deterministic Euler equation on its lifetime of existence. A particular challenge is that the anisotropic scaling ruins the divergence-free property for the spatial correlation functions of the noise.

\end{abstract}
	
\tableofcontents
\textcolor{white}{Hello}
\thispagestyle{empty}
\newpage

\setcounter{page}{1}

\section{Introduction} \label{section introduction}

This work concerns the inviscid limit of an anisotropic stochastic Navier-Stokes equation, as the horizontal and vertical components of the turbulent viscosity approach zero at different rates. We consider a divergence-free solution $u$ of the equation
\begin{equation} \label{some main equation} 
    du_t = - B(u_t, u_t) dt - \nu_h A_h u_t\, dt - \nu_{z} A_{z} u_t\, dt  - \nu_h^{\frac{1}{2}}\mathcal{P} \mathcal{G}^h u_t \circ d\mathcal{W}_t -  \nu_z^{\frac{1}{2}} \mathcal{P} \mathcal{G}^z u_t \circ d\mathcal{W}_t
\end{equation}
which represents the velocity of a fluid, posed on a domain $\mathscr{O} = \T^2 \times (0,1)$ and supplemented with the no-slip boundary condition $u = 0$ on $\partial \mathscr{O}$. Here $\mathcal{P}$ is the Leray Projector onto divergence-free vector fields with zero normal component on $\partial \mathscr{O}$, $B(u_t,u_t) = \mathcal{P}\left(\left(u_t \cdot \nabla\right)u_t\right)$ is the nonlinear convective term, $A_h = -\mathcal{P}\sum_{j=1}^2 \partial_j^2$ is the horizontal Stokes Operator with turbulent horizontal viscosity $\nu_h$ and $A_z = -\mathcal{P}\partial_3^2$ is the vertical Stokes Operator with turbulent vertical viscosity $\nu_z$. In the Stratonovich stochastic integrals, $\mathcal{W}$ denotes a Cylindrical Brownian Motion acted upon by operators $\mathcal{G}^h$, $\mathcal{G}^z$ in the sense that
$$\mathcal{G}^h u_t \circ d\mathcal{W}_t = \sum_{i=1}^\infty \mathcal{G}^h_i u_t \circ dW^i_t, \qquad \mathcal{G}^z u_t \circ d\mathcal{W}_t = \sum_{i=1}^\infty \mathcal{G}^z_i u_t \circ dW^i_t $$
where $(W^i)$ is a collection of independent standard Brownian Motions comprising $\mathcal{W}$, along with pre-assigned spatial correlation functions $(\xi_i)$ with respect to which $(\mathcal{G}^h_i)$, $(\mathcal{G}^z_i)$ are defined. Each $\xi_i$ is smooth, satisfies $\xi_i \cdot \underline{n} = 0$ on $\partial \mathscr{O}$ where $\underline{n}$ is the outward unit normal vector, but is not assumed to be divergence-free. The operators $\mathcal{G}^h_i$, $\mathcal{G}^z_i$ are defined by
\begin{equation} \label{lead to}\mathcal{G}_i^hu_t = \sum_{j=1}^2\left(\xi_i^j\partial_ju_t + u_t^j\nabla \xi_i^j\right), \qquad \mathcal{G}_i^zu_t = \xi_i^3\partial_3u_t + u_t^3\nabla \xi_i^3 \end{equation}
where the superscript $j$ denotes the $j^{\textnormal{th}}$ component mapping of the vector field. Broadly speaking, our main result is the construction of martingale weak solutions to (\ref{some main equation}) which converge, as $\nu_h$ and $\frac{\nu_z}{\nu_h}$ approach zero, to the strong solution $w$ of the deterministic Euler equation
\begin{equation} \label{Euler}
    \partial_tw = -B(w,w)
\end{equation}
posed on $\mathscr{O}$ with boundary condition $w \cdot \underline{n} = 0$. The convergence is in $L^2_{\omega}L^\infty_tL^2_x$, over the newly constructed probability space and on the lifetime of existence of $w$. The precise statement can be found in Theorem \ref{main result no rotation}.

\subsection{Deterministic Theory}

For a smooth bounded domain in two or three dimensions, whether or not weak solutions of the Navier-Stokes equations with no-slip boundary conditions converge, as the viscosity is taken to zero, to the strong solution of the Euler equation remains one of the outstanding problems of mathematical fluid mechanics. Positive results have been limited to very specific cases regarding analyticity of initial data or structure of the flow [\cite{lopes2008vanishing1}, \cite{lopes2008vanishing}, \cite{sammartino998zero}, \cite{sammartino1998zero}], whilst conditional results such as [\cite{kato1984remarks}, \cite{kelliher2007kato}, \cite{wang2001kato}] characterise the convergence by energy dissipation in a boundary layer which is not known to hold in general. We also mention that the limit is known if one replaces the no-slip boundary condition of the Navier-Stokes equations with a Navier friction condition [\cite{clopeau1998vanishing}, \cite{filho2005inviscid}, \cite{iftimie2006inviscid}, \cite{kelliher2006navier}, \cite{masmoudi2012uniform}].\\

In the absence of noise, the present model was considered in [\cite{masmoudi1998euler}]. Anisotropic viscosity, that is where $\nu_h \neq \nu_z$, is a classical feature of geophysical fluid mechanics. When the fluid is turbulent we speak of the viscosity not as the molecular kinematic viscosity but rather a turbulent viscosity, measured for example by the speed of diffusion of tracers. The Coriolis force amplifies horizontal diffusion, so that $\nu_z$ is much smaller than $\nu_h$. Whilst many aspects of a complete geophysical fluid model are neglected here, the limit as $\frac{\nu_z}{\nu_h}$ approaches zero is a useful setting to expose geophysical fluid phenomena. For a more complete discussion, we refer the reader to [\cite{chemin2006mathematical}, \cite{greenspan1969theory}, \cite{pedlosky2013geophysical}].\\

In the aforementioned [\cite{masmoudi1998euler}], Masmoudi proves precisely the result that we are aiming to establish in stochastic analogy; the convergence of all Leray-Hopf weak solutions of the Navier-Stokes equation with no-slip boundary condition, as $\nu_h$ and $\frac{\nu_z}{\nu_h}$ approach zero, to the strong solution of the Euler equation. We describe the simple heuristics of the problem; as there is no physical boundary in the horizontal direction, then the passage of $\nu_h \rightarrow 0$ is harmless and we are only concerned with the limit $\nu_z \rightarrow 0$ alongside its formation of the boundary layer at the horizontal plates. Following the work of Kato [\cite{kato1984remarks}] in establishing the equivalence between inviscid convergence and energy dissipation in a boundary layer, we anticipate that the problem boils down to showing that
\begin{equation} \label{in mind} \lim_{\nu_z \rightarrow 0} \nu_z \int_0^T\norm{\nabla u_s}^2_{L^2_{\Gamma_{\delta}}}ds = 0 \end{equation}
where $\Gamma_{\delta} = \T^2 \times \left[(0,\delta) \cup (1-\delta, 1)\right]$ is a boundary strip of width $\delta$ shrinking with $\nu_z$. However, more can be said. The fundamental problem in verifying the inviscid limit is the disparity between the boundary conditions, as the tangential component of $u$ is prescribed to be zero at the boundary whilst there are no restrictions on the tangential component of $w$. On the other hand, $u \cdot \underline{n} = 0  =w \cdot \underline{n}$ so the normal components of $u$ and $w$ match at the boundary. This suggests to us that the wild behaviour of $u$ in the boundary layer concerns only the normal derivative in tangential directions; it is therefore very interesting that Kato's criterion was extended by Wang, in [\cite{wang2001kato}] and with Temam in [\cite{temam1997behavior}], to a consideration of only \textit{tangential} derivatives in a boundary layer. This replaces the sufficient condition (\ref{in mind}) with
\begin{equation} \nonumber \lim_{\nu_z \rightarrow 0} \nu_z \int_0^T\sum_{j=1}^2\norm{\partial_j u_s}^2_{L^2_{\Gamma_{\delta}}}ds = 0 \end{equation}
and by rewriting
\begin{equation} \nonumber \nu_z \int_0^T\sum_{j=1}^2\norm{\partial_j u_s}^2_{L^2_{\Gamma_{\delta}}}ds = \left(\frac{\nu_z}{\nu_h}\right) \nu_h \int_0^T\sum_{j=1}^2\norm{\partial_j u_s}^2_{L^2_{\Gamma_{\delta}}}ds \end{equation}
then, as $\nu_h \int_0^T\sum_{j=1}^2\norm{\partial_j u_s}^2_{L^2_{\Gamma_{\delta}}}ds$ is bounded due to the energy inequality of Leray-Hopf solutions, if $\frac{\nu_z}{\nu_h}$ approaches zero the condition holds. Rigorously, the analysis of [\cite{masmoudi1998euler}] relies on the construction of a boundary corrector $\mathscr{B}$ such that $w + \mathscr{B}$ is zero on the boundary, whilst $\mathscr{B}$ is only supported near the boundary and of $L^2_x$ norm vanishing with $\nu_z$. An integration by parts is now facilitated in energy estimates on $u - w - \mathscr{B}$, and the result is achieved through a careful analysis of the many terms involved.

\subsection{Structure of the Noise} \label{subs structure}

There are several considerations to be made when introducing noise into this system. At the first level we ignore any coupling with the viscosity and ask what form the noise should take. We have chosen a transport-stretching noise, following the principle of \textit{Stochastic Advection by Lie Transport} introduced by Holm in [\cite{holm2015variational}]. This yields a Stratonovich integral in the velocity equation of fluid flow, given by 
$$\sum_{i=1}^\infty \mathcal{G}_i u_t \circ dW^i_t, \qquad  \mathcal{G}_iu_t = \sum_{j=1}^3\left(\xi_i^j\partial_ju_t + u_t^j\nabla \xi_i^j\right)$$
where the $(\xi_i)$ are spatial correlation functions of the driving noise as previously discussed. In [\cite{holm2015variational}] the noise is derived through geometric variational principles and is shown to preserve Kelvin's Circulation Theorem. This theory has been expanded upon across [\cite{crisan2022variational}, \cite{holm2021stochastic}, \cite{street2021semi}], and has run in tandem with derivations of transport noise in fluids through a Lagrangian Reynolds Decomposition and Transport Theorem given by M\'{e}min [\cite{memin2014fluid}], which has been further developed in [\cite{chapron2018large}, \cite{debussche2025variational}, \cite{resseguier2017geophysical}]. The theory is bolstered by numerical analysis and data assimilation presented throughout [\cite{chapron2024stochastic}, \cite{cotter2020data},  \cite{crisan2023implementation}, \cite{dufee2022stochastic}, \cite{ephrati2023data}] amongst many others. Stratonovich transport noise has also been derived following a stochastic model reduction scheme in [\cite{debussche2024second}, \cite{flandoli2022additive}, \cite{li2026navier}, \cite{luongo2026averaging}]. All of this recent progress supports the classical ideas of [\cite{brzezniak1992stochastic}, \cite{kraichnan1968small}, \cite{mikulevicius2001equations}, \cite{mikulevicius2004stochastic}], and we suggest [\cite{chapron2023stochastic}, \cite{flandoli2023stochastic}] for a review of the topic.\\

Secondly, we could choose for the limiting Euler equation to be stochastic or to investigate the vanishing noise limit with viscosity. Both regimes are of independent interest, however the former is immediately limited by the available well-posedness theory of the stochastic Euler equation on a bounded domain. Indeed, whilst local well-posedness is known for a Lipschitz multiplicative noise as shown in [\cite{glatt2012local}], or for transport noise on the torus demonstrated in [\cite{crisan2019solution}, \cite{goodair2025closed}], the case of transport noise on a bounded domain remains open. Weak solutions were proven to exist in [\cite{goodair2025navier}], however these lack the regularity required to conduct energy estimates akin to [\cite{masmoudi1998euler}]. Consequently, we only consider the regime of vanishing noise. The next task is to decide on a viscous scaling sending the noise to zero, for which we use $\nu^{\frac{1}{2}}$ having been motivated in [\cite{kuksin2004eulerian}] as the only noise scaling which leads to non-trivial limiting measures (in the limit $t \rightarrow \infty$ and $\nu \rightarrow 0$) for an additive noise in two dimensions in the absence of a boundary. The significance of this scaling for energy balance is further underlined in [\cite{kuksin2005family}, \cite{kuksin2006remarks}, \cite{kuksin2008distribution}] and has been used to study the inviscid limit problem in [\cite{brzezniak2026inviscid}, \cite{butori2024large}, \cite{glatt2015inviscid}, \cite{goodair2025zero}, \cite{luongo2024inviscid}].\\ 

The main novelty and challenge arises due to the third consideration, which is how to introduce the anisotropy into the noise scaling. As far as we are aware, such a problem has not been considered. Our approach is to split the spatial correlation functions $(\xi_i)$ of the driving Cylindrical Brownian Motion into their horizontal and vertical components, scaling by $\nu_h^{\frac{1}{2}}$ and $\nu_z^{\frac{1}{2}}$ respectively. To express this let us fix a generalised notation $\mathcal{G}$ as an operator on vector fields $\phi$, $f$ defined by
\begin{equation} \label{definition of G} \mathcal{G}_{\phi}f = \sum_{j=1}^3\left(\phi^j\partial_jf + f^j\nabla \phi^j\right)\end{equation}
such that $\mathcal{G}_i = \mathcal{G}_{\xi_i}$. The horizontal and vertical components of $\xi_i$ are given by
$$\xi_i^h = \left(\xi_i^1, \xi_i^2, 0\right), \qquad \xi_i^z = \left(0, 0, \xi_i^3 \right) $$
which we scale and combine for the effective spatial correlation function
\begin{equation} \label{effective spatial correlation} \tilde{\xi}_i = \nu_h^{\frac{1}{2}}\xi_i^h + \nu_z^{\frac{1}{2}}\xi_i^z = \left(\nu_h^{\frac{1}{2}}\xi_i^1, \nu_h^{\frac{1}{2}}\xi_i^2, \nu_z^{\frac{1}{2}}\xi_i^3\right).\end{equation}
For simplicity let us define $\tilde{\mathcal{G}}$ by
$\tilde{\mathcal{G}}_i = \mathcal{G}_{\tilde{\xi_i}}$, then our anisotropic stochastic Navier-Stokes equation reads as
\begin{equation} \label{some main equation 2} 
    du_t = - B(u_t, u_t) dt  - \nu_{h} A_{h} u_t\, dt - \nu_z A_z u_t\, dt - \mathcal{P} \tilde{\mathcal{G}} u_t \circ d\mathcal{W}_t. 
\end{equation}
The operator $\tilde{\mathcal{G}}_i$ has the explicit expression
\begin{align*}
  \tilde{\mathcal{G}}_iu_t &= \sum_{j=1}^3\left(\tilde{\xi}_i^j\partial_ju_t + u_t^j\nabla \tilde{\xi}_i^j\right)
  \\&= \nu_{h}^{\frac{1}{2}}\sum_{j=1}^2\left(\xi_i^j\partial_ju_t + u_t^j\nabla \xi_i^j\right) + \nu_{z}^{\frac{1}{2}}\left(\xi_i^3\partial_3u_t + u_t^3\nabla \xi_i^3\right)\\
  &= \nu_{h}^{\frac{1}{2}}\mathcal{G}_i^hu_t + \nu_{z}^{\frac{1}{2}}\mathcal{G}_i^zu_t
\end{align*}
for $\mathcal{G}_i^h$, $\mathcal{G}_i^z$ as defined in (\ref{lead to}). Thus, we arrive at equation (\ref{some main equation}). We remark that $\mathcal{G}_i^h = \mathcal{G}_{\xi_i^h}$ and $\mathcal{G}_i^z = \mathcal{G}_{\xi_i^z}$, hence the horizontal and vertical superscripting on $\mathcal{G}_i^h$ and $\mathcal{G}_i^z$ refers to transport and stretching along the horizontal and vertical components, respectively, of the spatial correlation functions. In addition, whilst we have chosen to incorporate the noise splitting into the operator $\mathcal{G}$, one could equivalently do this at the level of $\mathcal{W}$. Indeed if one considers
$$\mathcal{W} = \sum_{i=1}^\infty \xi_iW^i$$
 and decomposes into its horizontal and vertical components $\mathcal{W} = \mathcal{W}^h + \mathcal{W}^z$ where
\begin{equation} \label{franco equivalent}
    \mathcal{W}^h = \sum_{i=1}^\infty \xi_i^hW^i, \qquad \mathcal{W}^z = \sum_{i=1}^\infty \xi_i^zW^i,
\end{equation}
then one has the equivalent representation
$$\nu_h^{\frac{1}{2}}\mathcal{G}^h u_t \circ d\mathcal{W}_t +  \nu_z^{\frac{1}{2}} \mathcal{G}^z u_t \circ d\mathcal{W}_t = \nu_h^{\frac{1}{2}}\mathcal{G} u_t \circ d\mathcal{W}^h_t + \nu_z^{\frac{1}{2}} \mathcal{G} u_t \circ d\mathcal{W}^z_t.$$
The splitting (\ref{franco equivalent}) also appeared in [\cite{flandoli2024boussinesq}] for a Taylor-Proudman model with transport-stretching noise, a consequence of the 2D-3C nature of the equation.

\subsection{Aspects of the Proof} \label{subs aspects}

To motivate a discussion on the main elements of the proof, let us mention some existing results on inviscid limits for the stochastic Navier-Stokes equations. As a reminder, the precise statement of our result is given in Theorem \ref{main result no rotation}. Stochastic versions of Kato's Criterion have been established for a vanishing additive noise in 2D [\cite{luongo2024inviscid} ], vanishing transport type noise in 3D [\cite{goodair2025zero}], and non-vanishing additive noise in 2D [\cite{wang2024kato}]. For Navier boundary conditions in 2D, the limit for non-vanishing additive noise [\cite{cipriano2015inviscid}] and transport-stretching noise [\cite{goodair2025navier}] has been proven. Results on invariant measures in 2D for an additive noise without physical boundary are given in [\cite{bessaih2013inviscid}, \cite{brzezniak2026inviscid}, \cite{glatt2015inviscid}]. All of these results are for the isotropic case; with anisotropy we are only aware of one result, given in [\cite{wang2024zero}], dealing with a non-vanishing horizontal viscosity and a vertical viscosity vanishing as the rotation speed of an additional Coriolis term is taken to infinity. The authors consider an additive noise acting only in the two horizontal directions, and the limit is a damped 2D stochastic Navier-Stokes equation. Therefore, the structure of the noise in (\ref{some main equation}) and its inviscid limit is fundamentally new.\\

The first step in treating (\ref{some main equation}) is to convert the Stratonovich equation to It\^{o} form, where the analysis is much more favourable. One can see that the cross-variation will involve cross-terms between the horizontal and vertical viscosities. To make the conversion we use the compact expression (\ref{some main equation 2}), which by following [\cite{goodair2025stratonovich}] yields the equation 
\begin{equation} \label{some main equation 2 Ito} 
    du_t = - B(u_t, u_t)\ dt  - \nu_{h} A_{h} u_t\, dt - \nu_z A_z u_t\, dt +\frac{1}{2} \sum_{i=1}^\infty\mathcal{P}\tilde{\mathcal{G}}_i^2u_tdt -  \mathcal{P} \tilde{\mathcal{G}} u_t d\mathcal{W}_t 
\end{equation}
up to a `cost of a derivative'. Note that the It\^{o}-Stratonovich corrector has the form $(\mathcal{P}\tilde{\mathcal{G}}_i)^2u_t$ a priori, however we use the property that $\mathcal{P}\tilde{\mathcal{G}}_i = \mathcal{P}\tilde{\mathcal{G}}_i\mathcal{P}$ shown in [\cite{goodair20223d}] Lemma 2.7 to rewrite it in the form of (\ref{some main equation 2 Ito}). We stress that this property is a not consequence of $\tilde{\xi}_i$ being divergence-free, as indeed it is not here. In full, equation (\ref{some main equation 2 Ito}) reads as
\begin{align}
    \nonumber du_t = &- B(u_t, u_t) dt - \nu_h A_h u_t\, dt - \nu_{z} A_{z} u_t\, dt  - \nu_h^{\frac{1}{2}} \mathcal{P} \mathcal{G}^h u_t  d\mathcal{W}_t -  \nu_z^{\frac{1}{2}} \mathcal{P} \mathcal{G}^z u_t  d\mathcal{W}_t\\
    &+ \frac{\nu_h}{2}\sum_{i=1}^\infty \mathcal{P}\mathcal{G}^h_i\mathcal{G}^h_iu_t dt + \frac{\nu_h^{\frac{1}{2}}\nu_z^{\frac{1}{2}}}{2}\sum_{i=1}^\infty \mathcal{P}\left( \mathcal{G}^h_i\mathcal{G}^z_iu_t + \mathcal{G}^z_i\mathcal{G}^h_iu_t \right) dt + \frac{\nu_z}{2}\sum_{i=1}^\infty \mathcal{P}\mathcal{G}^z_i\mathcal{G}^z_iu_t dt. \label{expanded ito form}
\end{align}
Even the existence of martingale weak solutions to (\ref{expanded ito form}) is unclear. Typically, the key step in showing such existence for transport noise models is using the divergence-free property of $\xi_i$ to obtain that for vector fields $f,g$ and the inner product in $L^2_x$,
\begin{equation} \label{to be replaced}
    \inner{(\xi_i \cdot \nabla)f}{g} = -\inner{f}{(\xi_i \cdot \nabla) g}
\end{equation}
hence in energy estimates, as one meets the term combining the It\^{o}-Stratonvoich corrector and quadratic variation,
\begin{equation} \label{to be replaced 2}\inner{(\xi_i \cdot \nabla)(\xi_i \cdot \nabla)u}{u} + \norm{(\xi_i \cdot \nabla)u}^2 = -\inner{(\xi_i \cdot \nabla)u}{(\xi_i \cdot \nabla)u} + \norm{(\xi_i \cdot \nabla)u}^2 = 0.
\end{equation}
With the additional stretching term one has to work a bit more in controlling commutators, but eventually a bound by $\norm{u}^2$ is achieved primarily due to the same fact. Without the divergence-free property we cannot come to the same conclusion; whilst we could very happily assume that $\xi_i$ is divergence-free, this is ruined by the viscous scaling in the effective spatial correlation function $\tilde{\xi}_i$ defined in (\ref{effective spatial correlation}). To circumvent the issue one could recognise that it is sufficient to have the horizontal and vertical components of $\xi_i$ divergence-free, that is $\partial_1\xi_i^1 + \partial_2\xi_i^2 = 0 = \partial_3\xi_i^3.$ However $\xi_i$ must also satisfy the impermeability condition $\xi_i \cdot \underline{n} = 0$ on $\partial{\mathscr{O}}$, which simply says that $\xi_i^3 = 0$ on $\partial \mathscr{O} = \T^2 \times \left(\{0\} \cup \{1\} \right)$. Combining with the assumption $\partial_3\xi_i^3 = 0$ implies that $\xi_i^3 = 0$ everywhere in $\mathscr{O}$, hence the underlying Cylindrical Brownian Motion would only act in two dimensions so the phenomenon that we are investigating is trivialised.\\

The analysis of transport noise where the spatial correlation functions are not divergence-free appears completely absent in the literature. Even in the cases where transport noise is introduced into a compressible fluid [\cite{breit2022compressible}, \cite{crisan2024well}], the divergence-free assumption still appears to facilitate the analysis. We also mention that for the 2D-3C model considered in [\cite{flandoli2024boussinesq}] the authors assume that $\xi_i$ is divergence-free, which also implies that $\partial_1\xi_i^1 + \partial_2\xi_i^2 = 0 = \partial_3\xi_i^3$ as $\xi_i$ is independent of the third variable, however as their domain is the torus then $\xi_i^3$ does not need to be trivial.\\

We progress our analysis by recognising that in place of (\ref{to be replaced}), we have 
$$ \inner{(\tilde{\xi}_i \cdot \nabla)f}{g} = -\inner{f}{(\tilde{\xi}_i \cdot \nabla) g} -\inner{f}{\left(\sum_{j=1}^3\partial_j\tilde{\xi}_i^j\right)g} $$
and therefore, revisiting (\ref{to be replaced 2}),
\begin{align*}
    \inner{(\tilde{\xi}_i \cdot \nabla)(\tilde{\xi}_i \cdot \nabla)u}{u} &+ \norm{(\tilde{\xi}_i \cdot \nabla)u}^2\\  &= -\inner{(\tilde{\xi}_i \cdot \nabla)u}{\left(\sum_{j=1}^3\partial_j\tilde{\xi}_i^j\right)u}\\
    &\leq \left(\norm{\tilde{\xi}_i^h}_{L^\infty}\sum_{j=1}^2\norm{\partial_ju} + \norm{\tilde{\xi}_i^z}_{L^\infty}\norm{\partial_3u} \right)\norm{\tilde{\xi}_i}_{W^{1,\infty}}\norm{u}\\
    &\leq \left(\nu_h^{\frac{1}{2}}\norm{\xi_i^h}_{L^\infty}\sum_{j=1}^2\norm{\partial_ju} + \nu_z^{\frac{1}{2}}\norm{\xi_i^z}_{L^\infty}\norm{\partial_3u} \right)(\nu_h^{\frac{1}{2}} + \nu_z^{\frac{1}{2}})\norm{\xi_i}_{W^{1,\infty}}\norm{u}\\
    &\leq c_{\delta}(\nu_h + \nu_z)\norm{\xi_i}_{W^{1,\infty}}^2\norm{u}^2 + \delta\nu_h\norm{\xi_i^h}_{L^\infty}^2\sum_{j=1}^2\norm{\partial_ju}^2 + \delta\nu_z\norm{\xi_i^z}_{L^\infty}^2\norm{\partial_3u}^2
\end{align*}
by Young's Inequality with any small parameter $0 < \delta$. In particular $\delta$ can be chosen sufficiently small such that the latter two terms can be hidden in the viscous smoothing, and with $\delta$ fixed the first term will vanish with $\nu_h$ and $\nu_z$. Whilst this illustrates an important point, in the proof of our main result one must be much more precise in dealing with various cross-terms of the two directions.\\

An expected consequence of the lack of uniqueness for weak solutions of the 3D Navier-Stokes equations is that our solutions to (\ref{some main equation 2 Ito}) are only probabilistically weak, meaning that for every $\nu_h$ and $\nu_z$ we can find a probability space and Cylindrical Brownian Motion supporting a solution, however a priori the choice of that space may well depend on $\nu_h$ and $\nu_z$ themselves. To consider the expectation of the difference of the solutions we will need a single probability space supporting all solutions, which we construct by first fixing a sequence $(\nu_h^k, \nu_z^k)$ and then taking the infinite product of the probability spaces supporting the corresponding solutions $u^k$. Our main result is therefore stated for sequences of viscosities. Note that we do not obtain or need a uniform Cylindrical Brownian Motion. We proceed by considering $\mathbbm{E}\left(\sup_{t \in [0,T]}\norm{u^k_t - w_t - \mathscr{B}_t}_{L^2_x}^2 \right)$ where $\mathscr{B}$ is the boundary corrector from [\cite{masmoudi1998euler}]. To work with the evolution equation for $u^k$ we must look at the level of its Galerkin approximation and pass to the limit, given that the nonlinear term does not belong to $L^2_tH^{-1}_x$. Whilst the same is necessary in [\cite{masmoudi1998euler}] it is only mentioned as a formality, however this does not appear trivial. Strong solutions of the Euler equation $w$ are required as $H^{\gamma}$ bounds are used for it with $\frac{5}{2} < \gamma$, but if one passes the Galerkin projections $\mathcal{P}_n$ onto $w$ and tries to conclude by using $H^{\gamma}$ bounds on $\mathcal{P}_nw$ then this will fail as $\norm{\mathcal{P}_nw}_{H^{\gamma}}$ explodes with $n$ if $w \neq 0$ on $\partial \mathscr{O}$ (see [\cite{goodair2026high}] Appendix D). We must remove the projections before estimating term by term, generating quantities of the form $\norm{(I - \mathcal{P}_n)(w + \mathscr{B})}_{H^1_x}$ with some careful bounds, which do converge to zero as $w + \mathscr{B}$ satisfies the Dirichlet boundary condition.\\

We close this section with a brief comment on the necessity of features of the noise scaling for our method. Indeed if we were to scale  by $\nu^{\alpha}$ for some $0 < \alpha < \frac{1}{2}$, or to avoid decomposing the noise and suppose that the vertical component only decays with $\nu_h^{\frac{1}{2}}$, then our arguments would fail. This is a consequence of the $\nu_h^{\frac{1}{2}}, \nu_z^{\frac{1}{2}}$ scaling for the first order transport term matching the $\nu_h,\nu_z$ scaling of the Laplacian.

\section{Preliminaries}

\subsection{Functional Analytic Preliminaries} \label{subs funct anal}

We recall that $\mathscr{O} = \T^2 \times (0,1)$, and denote the usual Sobolev Spaces $W^{s,p}(\mathscr{O};\R^3)$, $H^{\gamma}(\mathscr{O};\R^3)$ by simply $W^{s,p}$, $H^{\gamma}$. We shall use $\inner{\cdot}{\cdot}$ to represent the $L^2$ inner product and similarly for the norm, whilst also employing subscripts $L^p_h$, $L^p_z$ as shorthand for $L^p\left(\T^2;\R^3\right)$ and $L^p\left((0,1);\R^3\right)$ respectively. This shorthand will also apply for general Euclidean target spaces, which shall be clear from the context. Let $C^{\infty}_{0,\sigma}$ be the space of smooth, compactly supported, divergence-free functions from $\mathscr{O}$ into $\R^3$. Then we define $L^2_{\sigma}$, $W^{1,2}_{\sigma}$ as the completion of $C^{\infty}_{0,\sigma}$ in $L^2$ and $W^{1,2}$ respectively; $W^{1,2}_{\sigma}$ is precisely the subspace of $W^{1,2}$ consisting of divergence-free and zero-trace functions, whilst $H^{\gamma} \cap L^2_{\sigma}$ for $1 \leq \gamma$ is the subspace of $H^{\gamma}$ of divergence-free functions $f$ satisfying $f \cdot \underline{n} = 0$ on $\partial \mathscr{O}$, where $\underline{n}$ is the outward unit normal vector at $\partial \mathscr{O}$. The geometry of the domain means that $f \cdot \underline{n} = 0$ is equivalent to $f^3 = 0$ on $\partial \mathscr{O}$. We recommend [\cite{temam2001navier}] for a proof of these facts.\\

Henceforth we fix a deterministic initial condition $u_0 \in H^{\gamma} \cap L^2_{\sigma}$ for some $\frac{5}{2} < \gamma$, as the classical regime for existence and uniqueness of local strong solutions of the Euler equation. Namely, referring to [\cite{bourguignon1974remarks}] Theorem 1 for example, there exists some $0 < T < \infty$ and a unique $w \in C\left([0,T]; H^{\gamma} \cap L^2_{\sigma} \right) \cap C^1\left([0,T] \times \bar{\mathscr{O}}; \R^3 \right)$ such that the identity \begin{equation} \label{euler identity} w_t = u_0 - \int_0^tB(w_s,w_s)ds \end{equation} holds for all $0 \leq t \leq T$ in $L^2_{\sigma}$. We recall the boundary corrector function constructed in [\cite{masmoudi1998euler}] Subsection 2.1.

\begin{lemma} \label{masmoudi corrector}
Let $0 < \theta$ be an arbitrary parameter and fix $0 < \nu_z$. There exists a function $\mathscr{B} \in C\left([0,T]; H^{\gamma} \cap L^2_{\sigma} \right)$ such that $w_t + B_t \in W^{1,2}_{\sigma}$ for every $t \in [0,T]$, of the form
$$\mathscr{B}_t(x,y,z) = \mathscr{M}(z)\begin{pmatrix}
       w^1_t(x,y,0) + w^1_t(x,y,1)\\
       w^2_t(x,y,0) + w^2_t(x,y,1)\\
       \partial_3w^3_t(x,y,0) + \partial_3w^3_t(x,y,1)
   \end{pmatrix} \coloneqq \mathscr{M}(z)\mathscr{A}_t(x,y)$$
where $\mathscr{M}(z) \in \R^{3 \times 3}$ and $\mathscr{M}$ is only supported on $\left[0, \frac{1}{4}\right] \cup \left[\frac{3}{4}, 1\right]$. Moreover the matrix valued function $\mathscr{M}$ is composed of just an upper left matrix $\mathscr{M}^h \in \R^{2 \times 2}$ and a lower right scalar $\mathscr{M}^z \in \R$, satisfying the bounds
\begin{align*}
    \norm{\mathscr{M}}_{L^2_z} + \norm{z\partial_3\mathscr{M}(z)}_{L^2_{\left(0,\frac{1}{4}\right)}} + \norm{(1-z)\partial_3\mathscr{M}(z)}_{L^2_{\left(\frac{3}{4}, 1\right)}} &\leq c\left(\theta \nu \right)^{\frac{1}{4}}\\
    \norm{\mathscr{M}^z}_{L^\infty_z} + \norm{z^2\partial_3\mathscr{M}(z)}_{L^\infty_{\left(0,\frac{1}{4}\right)}} + \norm{(1-z)^2\partial_3\mathscr{M}(z)}_{L^\infty_{\left(\frac{3}{4},1\right)}}
     &\leq c\left(\theta \nu \right)^{\frac{1}{2}}\\
    \norm{\partial_3\mathscr{M}}_{L^2_z} &\leq c\left(\theta \nu \right)^{-\frac{1}{4}}\\
    \norm{\mathscr{M}}_{L^\infty_z} + \norm{\partial_3\mathscr{M}^z}_{L^\infty_z} &\leq c
\end{align*}
for a constant $c$ independent of $\theta$ and $\nu_z$.



\end{lemma}

We remark that the boundary corrector from [\cite{masmoudi1998euler}] is constructed for the domain $\T^2 \times [0,\infty)$ so only one horizontal plate is considered. On $\mathscr{O}$ we first build the Masmoudi corrector near the bottom horizontal plate $\T^2 \times \{0\}$ with support in $\T^2 \times [0,\frac{1}{4}]$, and then identically establish the corrector near  $\T^2 \times \{1\}$ with support in $\T^2 \times [\frac{3}{4},1]$ by replacing $z$ with $1-z$. Our $\mathscr{B}$ is then the sum of these correctors. We also establish bounds on $\mathscr{A}$ below.


\begin{lemma} \label{bounds on A}
The function $\mathscr{A}$ from Lemma \ref{masmoudi corrector} satisfies the bounds
\begin{align*}
   \norm{\mathscr{A}_s}_{L^\infty_h} + \sum_{j=1}^2\norm{\partial_j\mathscr{A}_s}_{L^2_h} + \norm{\partial_t\mathscr{A}_s}_{L^2_h} \leq c\norm{w_s}_{H^{\gamma}}.
\end{align*}
for some constant $c$ independent of $w$.

\end{lemma}

\begin{proof}
Control on the first term follows simply from the fact that
$$\norm{\mathscr{A}_s}_{L^\infty_h} \leq c\norm{w_s}_{W^{1,\infty}} \leq c\norm{w_s}_{H^{\gamma}}$$
by the usual Sobolev Embedding $H^{\gamma} \hookrightarrow W^{1,\infty}$. For the second we observe that
$$\norm{\partial_j\mathscr{A}_s}_{L^2_h} \leq c\norm{w_s}_{H^2_{\partial\mathscr{O}}} \leq c\norm{w_s}_{H^{\gamma}}$$
by the trace inequality, as $\frac{1}{2} < \gamma -2$. In the third term,
$$\norm{\partial_t\mathscr{A}_s}_{L^2_h} \leq c\norm{\partial_tw_s}_{H^1_{\partial\mathscr{O}}} \leq c\norm{\partial_tw_s}_{H^{\gamma -1}}$$
using the trace inequality again as $\frac{1}{2} < (\gamma-1)-1$. Now we can substitute in the explicit form $\partial_tw_s = B(w_s,w_s)$, and using that $H^{\gamma-1}$ is an algebra as $\frac{3}{2} < \gamma -1$ the result swiftly follows.
    
\end{proof}

We conclude this subsection with a final inequality to be used in the main proof.

\begin{lemma} \label{Hardy Littlewood}
For $1 < p \leq \infty$, let $f\in W^{1,p}(\mathscr{O};\R)$ be such that $f = 0$ on $\partial \mathscr{O}$. Then
$$\norm{\frac{f(x,y,z)}{z}}_{L^p} + \norm{\frac{f(x,y,z)}{1-z}}_{L^p} \leq c\norm{\partial_3f}_{L^p}$$
for some constant $c$ independent of $f$.

\end{lemma}

\begin{proof}
    Starting with the first term, then $f = 0$ on $\partial \mathscr{O}$ implies that $f(x,y,0) = 0$ and as $f\in W^{1,p}(\mathscr{O};\R)$ we may write
    $$f(x,y,z) = \int_0^z\partial_3f(x,y,\eta)d\eta.$$
Suppressing the dependence on $x,y$, note that
\begin{align*}
    \abs{\frac{f(z)}{z}} \leq \frac{1}{z}\int_0^z\abs{\partial_3f(\eta)}d\eta \leq \frac{1}{z}\int_0^{2z}\abs{\partial_3f(\eta)}d\eta \leq \sup_{0 < r}\left(\frac{1}{r}\int_{z-r}^{z + r}\abs{\partial_3f(\eta)}d\eta\right)
\end{align*}
and the Hardy-Littlewood Maximal Inequality gives us that
\begin{align*}
    \norm{ \sup_{0 < r}\left(\frac{1}{r}\int_{z-r}^{z + r}\abs{\partial_3f(\eta)}d\eta\right)}_{L^p_z} \leq c\norm{\partial_3f}_{L^p_z}.
\end{align*}
Therefore
$$ \norm{\frac{f(x,y,z)}{z}}_{L^p} = \norm{ \norm{\frac{f(x,y,z)}{z}}_{L^p_z}}_{L^p_h} \leq c\norm{\norm{\partial_3f}_{L^p_z}}_{L^p_h} = c\norm{\partial_3f}_{L^p}$$
as required. The second term is treated symmetrically, now using that $f(x,y,1) = 0$ so
$$f(x,y,z) = -\int_z^1\partial_3f(x,y,\eta)d\eta$$
and 
\begin{align*}
    \abs{\frac{f(z)}{1-z}} \leq \frac{1}{1-z}\int_z^1\abs{\partial_3f(\eta)}d\eta \leq \frac{1}{1-z}\int_{2z-1}^{1}\abs{\partial_3f(\eta)}d\eta \leq \sup_{0 < r}\left(\frac{1}{r}\int_{(1-z)-r}^{(1-z) + r}\abs{\partial_3f(\eta)}d\eta\right)
\end{align*}
leading to the same conclusion, as
$$\norm{ \sup_{0 < r}\left(\frac{1}{r}\int_{z-r}^{z + r}\abs{\partial_3f(\eta)}d\eta\right)}_{L^p_z} = \norm{ \sup_{0 < r}\left(\frac{1}{r}\int_{(1-z)-r}^{(1-z) + r}\abs{\partial_3f(\eta)}d\eta\right)}_{L^p_z}.$$


\end{proof}






\subsection{Stochastic Preliminaries} \label{subs stoch prelim}

By a filtered probability space $\mathcal{S}$, we mean a quartet $(\Omega,\mathcal{F},(\mathcal{F}_t), \mathbbm{P})$ satisfying the usual conditions of completeness and right continuity. Let us fix an auxiliary Hilbert Space $\mathfrak{U}$ with orthonormal basis $(e_i)$. We say that $\mathcal{W}$ is a Cylindrical Brownian Motion over $\mathfrak{U}$ with respect to $\mathcal{S}$ if $\mathcal{W}_t = \sum_{i=1}^\infty e_iW^i_t$ as a limit in $L^2(\Omega;\mathfrak{U}')$ for some collection $(W^i)$ of i.i.d. standard real valued Brownian Motions with respect to $\mathcal{S}$, and $\mathfrak{U}'$ an enlargement of the Hilbert Space $\mathfrak{U}$ such that the embedding $J: \mathfrak{U} \rightarrow \mathfrak{U}'$ is Hilbert-Schmidt. In this case $\mathcal{W}$ is a $JJ^*-$Cylindrical Brownian Motion over $\mathfrak{U}'$. Given a process $F:[0,T] \times \Omega \rightarrow \mathscr{L}^2(\mathfrak{U};\mathscr{H})$ progressively measurable and such that $F \in L^2\left(\Omega \times [0,T];\mathscr{L}^2(\mathfrak{U};\mathscr{H})\right)$, for any $0 \leq t \leq T$ we define the stochastic integral $$\int_0^tF_sd\mathcal{W}_s \coloneqq \sum_{i=1}^\infty \int_0^tF_s(e_i)dW^i_s,$$ where the infinite sum is taken in $L^2(\Omega;\mathscr{H})$. We can extend this notion to processes $F$ which are such that $F(\omega) \in L^2\left( [0,T];\mathscr{L}^2(\mathfrak{U};\mathscr{H})\right)$ for $\mathbbm{P}-a.e.$ $\omega$ via the traditional localisation procedure. In this case the stochastic integral is a local martingale in $\mathscr{H}$. We thus consider $\mathcal{G}$ and its relatives as an operator on $\mathfrak{U}$ by $\mathcal{G}(e_i) = \mathcal{G}_i$, see [\cite{goodair2024stochastic}] Subchapter 3.2. We defer to [\cite{goodair2024stochastic}] Chapter 2 for further details on this construction and properties of the stochastic integral. We shall make use of the Burkholder-Davis-Gundy Inequality ([\cite{da2014stochastic}] Theorem 4.36) and the energy identity ([\cite{goodair2024stochastic}] Proposition 4.3, [\cite{liu2015stochastic}] Theorem 4.2.5).

\subsection{Transport-Stretching Noise} \label{subs transport stretch}

We now address key properties of the anisotropically scaled transport-stretching noise operator, starting by recalling the definition of $\mathcal{G}$ from (\ref{definition of G}) for vector fields $\phi$, $f$ as
$$\mathcal{G}_{\phi}f = \sum_{j=1}^3\left(\phi^j\partial_jf + f^j\nabla \phi^j\right).$$

\begin{lemma} \label{unnamed}
    Let $\phi \in H^1$ satisfy that $\phi \cdot \underline{n} = 0$ on $\partial \mathscr{O}$. Then for all $f,g \in H^1$, we have that
    \begin{align*}
        \inner{\mathcal{L}_{\phi}f}{g} &= -\inner{f}{\mathcal{L}_{\phi}g} - \inner{f}{\left(\sum_{j=1}^3\partial_j\phi^j\right)g},\\
        \inner{\mathcal{T}_{\phi}f}{g} &= \inner{f}{\mathcal{L}_g\phi}.
    \end{align*}
Therefore, we define the operators $\mathcal{L}_{\phi}^*$, $\mathcal{T}_{\phi}^*$ and $\mathcal{G}_{\phi}^*$ on $H^1$ by
\begin{align*}
    \mathcal{L}_{\phi}^*f &= -\mathcal{L}_{\phi}f - \left(\sum_{j=1}^3\partial_j\phi^j\right)f,\\
    \mathcal{T}_{\phi}^*f &= \mathcal{L}_{f}\phi,\\
    \mathcal{G}_{\phi}^*f &= \mathcal{L}_{\phi}^*f + \mathcal{T}_{\phi}^*f.
\end{align*}
    
\end{lemma}

\begin{proof}
We observe that
       \begin{align*}
    \inner{\mathcal{L}_{\phi} f}{g} &= \sum_{j=1}^3\sum_{l=1}^3\inner{\phi^j\partial_jf^l}{g^l}\\
    &= \sum_{j=1}^3\sum_{l=1}^3\left(\inner{\phi^j\partial_jf^l}{g^l} + \inner{\partial_j\phi^jf^l}{g^l} - \inner{\partial_j\phi^jf^l}{g^l}\right)\\
    &= \sum_{j=1}^3\sum_{l=1}^3\left(\inner{\partial_j(\phi^jf^l)}{g^l} - \inner{\partial_j\phi^jf^l}{g^l}\right)\\
    &= -\sum_{j=1}^3\sum_{l=1}^3\inner{\phi^jf^l}{\partial_jg^l} + \sum_{j=1}^3\sum_{l=1}^3\inner{\phi^jf^l}{g^l\underline{n}^j}_{L^2(\partial \mathscr{O};\R)} - \sum_{l=1}^3\inner{f^l}{\left(\sum_{j=1}^3\partial_j\phi^j\right)g^l}\\
    &= -\inner{f}{\mathcal{L}_{\phi} g} - \inner{f}{\left(\sum_{j=1}^3\partial_j\phi^j\right)g}
\end{align*}
    where we have used that $\sum_{j=1}^3\phi^j \underline{n}^j = \phi \cdot \underline{n} = 0$. For the second result, we simply observe that
    \begin{align*}
        \inner{\mathcal{T}_{\phi}f}{g} = \sum_{j=1}^3\sum_{l=1}^3\inner{f^j\partial_l\phi^j}{g^l} = \sum_{j=1}^3\sum_{l=1}^3\inner{f^j}{g^l\partial_l\phi^j} = \sum_{j=1}^3\inner{f^j}{\mathcal{L}_g\phi^j} = \inner{f}{\mathcal{L}_g\phi}.
    \end{align*}
\end{proof}

Let us now fix the spatial correlation functions $(\xi_i)$, satisfying $\xi_i \cdot \underline{n} = 0$ on $\partial \mathscr{O}$ and\\ $\sum_{i=1}^\infty \norm{\xi_i}_{W^{2,\infty}}^2 < \infty$. We further recall and extend several notations from Subsection \ref{subs structure}, firstly the splitting and scaling 
\begin{align*}
    \xi_i^h = \left(\xi_i^1, \xi_i^2, 0\right), \qquad \xi_i^z = \left(0, 0, \xi_i^3 \right), \qquad \tilde{\xi}_i^h = \nu_h^{\frac{1}{2}}\xi_i^h, \qquad \tilde{\xi}_i^z = \nu_z^{\frac{1}{2}}\xi_i^z
\end{align*}
with
$$ \tilde{\xi}_i = \tilde{\xi}_i^h + \tilde{\xi}_i^z = \left(\nu_h^{\frac{1}{2}}\xi_i^1, \nu_h^{\frac{1}{2}}\xi_i^2, \nu_z^{\frac{1}{2}}\xi_i^3\right)$$
and then the corresponding operators
$$\tilde{\mathcal{G}}_i = \mathcal{G}_{\tilde{\xi}_i}, \qquad \tilde{\mathcal{G}}^h_i = \mathcal{G}_{\tilde{\xi}_i^h}, \qquad \tilde{\mathcal{G}}_i^z = \mathcal{G}_{\tilde{\xi}_i^z}. $$
The property $\phi \cdot \underline{n} = 0$ is preserved for all variants of $\xi_i$ considered, as it is equivalent to $\phi^3 = 0$ on $\partial \mathscr{O}$. We thus obtain expressions for $\tilde{\mathcal{G}}_i^*$, $\tilde{\mathcal{G}}_i^{h,*}$ and $\tilde{\mathcal{G}}_i^{z,*}$ due to Lemma \ref{unnamed}. Noting that $\mathcal{L}_{\tilde{\xi}_i^h}f = \sum_{j=1}^2\tilde{\xi}_i^j\partial_jf$ and $\mathcal{L}_{\tilde{\xi}_i^z}f = \tilde{\xi}_i^3\partial_3f$, one deduces the following bounds.

\begin{lemma} \label{bounds on split noise}
There exists a constant $c$ such that, for all $f \in H^1$,
\begin{align*}
    \norm{\tilde{\mathcal{G}}_i^hf} + \norm{\tilde{\mathcal{G}}_i^{h,*}f} &\leq c\norm{\tilde{\xi}_i^h}_{W^{1,\infty}}\left(\sum_{j=1}^2\norm{\partial_jf} + \norm{f}\right),\\
     \norm{\tilde{\mathcal{G}}_i^zf} + \norm{\tilde{\mathcal{G}}_i^{z,*}f} &\leq c\norm{\tilde{\xi}_i^z}_{W^{1,\infty}}\left(\norm{\partial_3f} + \norm{f}\right).
\end{align*}
    
\end{lemma}

We also prove the key estimate discussed in Subsection \ref{subs aspects}.
\begin{lemma}
    For any parameter $0 < \delta$, there exists a constant $c_{\delta}$ such that, for all $f\in H^2$,
        \begin{align}
   \inner{\tilde{\mathcal{G}}_i^2f}{f} +  \norm{\tilde{\mathcal{G}}_if}^2  &\leq c_{\delta}\norm{\tilde{\xi}_i}_{W^{1,\infty}}^2\norm{f}^2 + \delta\norm{\tilde{\xi}_i^h}_{L^{\infty}}^2\sum_{j=1}^2\norm{\partial_jf}^2 +  \delta\norm{\tilde{\xi}_i^z}_{L^{\infty}}^2\norm{\partial_3f}^2 \label{bigbound1NEW} ,\\
    \inner{\tilde{\mathcal{G}}_if}{f}^2 &\leq c\norm{\tilde{\xi}_i}^2_{W^{1,\infty}}\norm{f}^4. \label{bigbound2NEW}
\end{align}
\end{lemma}

\begin{proof}
    In the direction of (\ref{bigbound1NEW}), we have that
    \begin{align} \nonumber
        \inner{\tilde{\mathcal{G}}_i^2f}{f} +  \norm{\tilde{\mathcal{G}}_if}^2 = \inner{\tilde{\mathcal{G}}_if}{\tilde{\mathcal{G}}_i^*f + \tilde{\mathcal{G}}_if} &= \inner{\tilde{\mathcal{G}}_if}{\mathcal{T}_{\tilde{\xi}_i}^*f + \mathcal{T}_{\tilde{\xi}_i}f -\left(\sum_{j=1}^3 \partial_j\tilde{\xi}_i^j\right)f}\\ &\leq c\norm{\tilde{\mathcal{G}}_if}\norm{\tilde{\xi}_i}_{W^{1,\infty}}\norm{f}. \label{pluggers}
    \end{align}
We split $\tilde{\mathcal{G}}_if$ into $\mathcal{L}_{\tilde{\xi}_i^h}f +\mathcal{L}_{\tilde{\xi}_i^z}f + \mathcal{T}_{\tilde{\xi}_i}f $, from which we have that
$$\norm{\tilde{\mathcal{G}}_if} \leq c\left(\sum_{j=1}^2\norm{\tilde{\xi}_i^h}_{L^{\infty}}\norm{\partial_jf} + \norm{\tilde{\xi}_i^z}_{L^{\infty}}\norm{\partial_3f} + \norm{\tilde{\xi}_i}_{W^{1,\infty}}\norm{f}\right).$$
Plugging this bound into (\ref{pluggers}) and applying Young's Inequality in the first two terms allows us to conclude (\ref{bigbound1NEW}). As for (\ref{bigbound2NEW}), observe that
\begin{align*}
    \inner{\tilde{\mathcal{G}}_if}{f}^2 \leq 2\inner{\mathcal{L}_{\tilde{\xi}_i}f}{f}^2 + 2\inner{\mathcal{T}_{\tilde{\xi}_i}f}{f}^2 \leq 2\inner{\mathcal{L}_{\tilde{\xi}_i}f}{f}^2 + c\norm{\tilde{\xi}_i}^2_{W^{1,\infty}}\norm{f}^4.
\end{align*}
One must deal with the transport term, for which we see that
\begin{align*}
    \inner{\mathcal{L}_{\tilde{\xi}_i}f}{f} = -\inner{\mathcal{L}_{\tilde{\xi}_i}f}{f} - \inner{f}{\left(\sum_{j=1}^3\partial_j\tilde{\xi}_i^j\right)f}
\end{align*}
hence
$$ \inner{\mathcal{L}_{\tilde{\xi}_i}f}{f}^2 = \frac{1}{4}\inner{f}{\left(\sum_{j=1}^3\partial_j\tilde{\xi}_i^j\right)f}^2 \leq c\norm{\tilde{\xi}_i}^2_{W^{1,\infty}}\norm{f}^4$$
    as required.
\end{proof}

The constants in the above lemmas are of course independent of $\xi_i$, $\nu_h$ and $\nu_z$.

\section{Anisotropic Inviscid Limit}

\subsection{Definitions and Main Result}

 We shall work on the interval $[0,T]$ for the duration of this section, where $T$ was the (finite) lifetime of existence of the Euler equation given in Subsection \ref{subs funct anal}. Let us begin by defining the notion of a martingale weak solution of the equation (\ref{expanded ito form}). We shall use the more compact representation (\ref{some main equation 2 Ito}).

\begin{definition} \label{definitionofspacetimeweakmartingale}
Let $\mathcal{S}$ be a filtered probability space. A pair $(u, \mathcal{W})$, where $\mathcal{W}$ is a Cylindrical Brownian Motion over $\mathfrak{U}$ with respect to $\mathcal{S}$, and $u$ is a progressively measurable process in $W^{1,2}_{\sigma}$ such that  $u \in  L^{\infty}\left([0,T];L^2_{\sigma}\right) \cap L^2\left([0,T];W^{1,2}_{\sigma}\right)$ $\mathbbm{P}-a.s.$, is said to be a martingale weak solution of (\ref{expanded ito form}) with respect to $\mathcal{S}$ if the identity
\begin{align*}
     \inner{u_t}{\phi} = \inner{u_0}{\phi} &- \int_0^t\inner{B(u_s, u_s)}{\phi} ds  - \nu_{h}\int_0^t \sum_{j=1}^2\inner{\partial_j u_s}{\partial_j \phi} ds - \nu_z\int_0^t\inner{\partial_3 u_s}{\partial_3\phi} ds\\ &+\frac{1}{2} \int_0^t\sum_{i=1}^\infty\inner{\tilde{\mathcal{G}}_iu_s}{\tilde{\mathcal{G}}_i^*\phi}ds -  \int_0^t\inner{\mathcal{P} \tilde{\mathcal{G}} u_s}{\phi} d\mathcal{W}_s
\end{align*}
holds for every $\phi \in W^{1,2}_{\sigma}$, $\mathbbm{P}-a.s.$ in $\R$, for all $t \in [0,T]$.
\end{definition}

Having established the definition, we can now state the main result. We recall the relevant regularities of $u_0 \in H^{\gamma} \cap L^2_\sigma$ for some $\frac{5}{2} < \gamma$, each $\xi_i \cdot \underline{n} = 0$ on $\partial\mathscr{O}$ with $\sum_{i=1}^\infty\norm{\xi_i}_{W^{2,\infty}}^2 < \infty$, and $w$ as given by (\ref{euler identity}).

\begin{theorem} \label{main result no rotation}
    Let $(\nu_h^k)$, $(\nu_z^k)$ be any two sequences of positive constants such that $(\nu_h^k)$ and $\left(\frac{\nu_z^k}{\nu_h^k} \right)$ converge to zero. Then there exists a filtered probability space $\mathcal{S}$ and a corresponding sequence of martingale weak solutions $\left(u^k, \mathcal{W}^k\right)$ of (\ref{expanded ito form}) with respect to $\mathcal{S}$ such that:
    \begin{enumerate}
        \item $(u^k) \longrightarrow w$ in $L^2\left(\Omega; L^\infty\left([0,T];L^2_{\sigma} \right) \right)$ as $k \longrightarrow \infty$;
        \item $\left(\nu_h^k\left(\sum_{j=1}^2 \partial_j u^k\right)\right)$ and $\left(\nu_z^k \partial_3u^k \right)$ $\longrightarrow 0$ in $L^2\left(\Omega; L^2\left([0,T];L^2_{\sigma} \right) \right)$ as $k \longrightarrow \infty$.
    \end{enumerate}
\end{theorem}

\subsection{Selection of Martingale Weak Solutions} \label{subs selection martingale weak}

In this subsection, we give the construction of martingale weak solutions used in the proof of Theorem \ref{main result no rotation}. In the isotropic case, that is where $\nu_h = \nu_z$, the existence of martingale weak solutions was proven in [\cite{goodair2024weak}] Theorem 5.1. The anisotropy does not disturb the proof in a significant way, by representing the noise term as in (\ref{some main equation 2 Ito}) then one only suffers the worse bound (\ref{bigbound1NEW}) which is entirely sufficient for the application in [\cite{goodair2024weak}] by combining with the viscous term. Note also that bounds in $W^{1,2}_{\sigma}$ are obtained by simply using the smaller coefficient $\nu_z$. We have the following.

\begin{proposition} \label{basic existence martingale weak}
    Let $0 < \nu_h, \nu_z$. Then there exists a filtered probability space $\mathcal{S}$ and a pair $(u, \mathcal{W})$ which is a martingale weak solution of (\ref{expanded ito form}) with respect to $\mathcal{S}$. Moreover, $(u, \mathcal{W})$ is obtained as the $n \rightarrow \infty$ limit of Galerkin Approximations $(u^n, \mathcal{W}^n)$ satisfying
    \begin{align*}
        u^n_t = \mathcal{P}_nu_0 &- \int_0^t\mathcal{P}_nB(u^n_s, u^n_s)\ ds  - \nu_{h}\int_0^t \mathcal{P}_nA_{h} u^n_s\, ds - \nu_z\int_0^t \mathcal{P}_nA_z u^n_s\, ds\\ &+\frac{1}{2}\int_0^t \sum_{i=1}^\infty\mathcal{P}_n\mathcal{P}\tilde{\mathcal{G}}_i^2u^n_sds -  \int_0^t\mathcal{P}_n\mathcal{P} \tilde{\mathcal{G}} u^n_s d\mathcal{W}^n_s.
    \end{align*}
Precisely, $(u^n) \rightarrow u$ in the weak* topology of $L^2\left(\Omega; L^\infty\left([0,T];L^2_{\sigma}\right) \right)$ and the weak topology of $L^2\left(\Omega; L^2\left([0,T];W^{1,2}_{\sigma}\right) \right)$.
    
\end{proposition}

The Galerkin Approximations will be used in the proof of Theorem \ref{main result no rotation} as one cannot look at the $L^2$ norm of weak solutions directly. Whilst Proposition \ref{basic existence martingale weak} informs us that for any given viscosity we can build a filtered probability space supporting a martingale weak solution, the choice of that space may very well depend on $\nu_h, \nu_z$ themselves. We now construct a space supporting all solutions in the anisotropic inviscid limit.\\

Let us fix any two sequences of positive constants $(\nu_h^k)$, $(\nu_z^k)$ such that $(\nu_h^k)$ and $\left(\frac{\nu_z^k}{\nu_h^k} \right)$ converge to zero, as in Theorem \ref{main result no rotation}. By Proposition \ref{basic existence martingale weak}, for each $k$ there exists a filtered probability space $\mathcal{S}^k \coloneqq \left(\tilde{\Omega}^k, \tilde{\mathcal{F}}^k, (\tilde{\mathcal{F}}^k_t), \tilde{\mathbbm{P}}^k \right)$ and a pair $(\tilde{u}^k, \tilde{\mathcal{W}}^k)$ which is a martingale weak solution of (\ref{expanded ito form}) with respect to $\mathcal{S}^k$. We now define the standard infinite dimensional product space $$\Omega:= \bigtimes_{k=0}^\infty \tilde{\Omega}^k, \quad \mathcal{F}:= \bigotimes_{k=0}^\infty \tilde{\mathcal{F}}^k, \quad \mathcal{F}_t:= \bigotimes_{k=0}^\infty \tilde{\mathcal{F}}^k_t, \quad \mathbbm{P}:= \bigtimes_{k=0}^\infty \tilde{\mathbbm{P}}^k$$
    such that $\mathcal{S} \coloneqq \left(\Omega,\mathcal{F},(\mathcal{F}_t), \mathbbm{P}\right)$ is a filtered probability space. Furthermore we introduce the component projections $(\mathscr{P}^k)$, $\mathscr{P}^k: \Omega \rightarrow \tilde{\Omega}^k$ and subsequently define $(u^k, \mathcal{W}^k)$ by $$u^k \coloneqq \tilde{u}^k\mathscr{P}^k, \qquad \mathcal{W}^k \coloneqq \tilde{\mathcal{W}}^k\mathscr{P}^k.$$
By construction, for each $k$, $(u^k, \mathcal{W}^k)$ is a martingale weak solution of (\ref{expanded ito form}) with respect to $\mathcal{S}$. 

\subsection{Proof of the Main Result}

This subsection is dedicated to the proof of Theorem \ref{main result no rotation}.

\begin{proof}[Proof of Theorem \ref{main result no rotation}:]

We choose the filtered probability space $\mathcal{S}$ and martingale weak solutions $(u^k, \mathcal{W}^k)$ as described in Subsection \ref{subs selection martingale weak}. Let us first approach the convergence $(u^k) \rightarrow w$ in $L^2\left(\Omega; L^\infty\left([0,T];L^2_{\sigma} \right) \right)$. Without loss of generality we suppose throughout that $\nu_h^k, \nu_z^k < 1$. Note that
$$\mathbbm{E}\left(\norm{u^k - w}^2_{L^\infty\left([0,T];L^2_{\sigma} \right)}\right) = \mathbbm{E}\left(\norm{\tilde{u}^k\mathscr{P}^k - w}^2_{L^\infty\left([0,T];L^2_{\sigma} \right)}\right) = \tilde{\mathbbm{E}}^k\left(\norm{\tilde{u}^k - w}^2_{L^\infty\left([0,T];L^2_{\sigma} \right)}\right) $$
where $\tilde{\mathbbm{E}}^k$ denotes the expectation over $\tilde{\Omega}^k$ with respect to $\tilde{\mathbbm{P}}^k$. As alluded to, for the computations we will use the Galerkin Approximation specified in Proposition \ref{basic existence martingale weak}. More precisely, there is an approximate sequence of solutions $(\tilde{u}^{k,n}, \tilde{\mathcal{W}}^{k,n})$ satisfying
\begin{align*}
        \tilde{u}^{k,n}_t = \mathcal{P}_nu_0 &- \int_0^t\mathcal{P}_nB(\tilde{u}^{k,n}_s, \tilde{u}^{k,n}_s)\ ds  - \nu^k_{h}\int_0^t \mathcal{P}_nA_{h} \tilde{u}^{k,n}_s\, ds - \nu^k_z\int_0^t \mathcal{P}_nA_z \tilde{u}^{k,n}_s\, ds\\ &+\frac{1}{2}\int_0^t \sum_{i=1}^\infty\mathcal{P}_n\mathcal{P}\tilde{\mathcal{G}}_i^2\tilde{u}^{k,n}_sds -  \int_0^t\mathcal{P}_n\mathcal{P} \tilde{\mathcal{G}} \tilde{u}^{k,n}_s d\tilde{\mathcal{W}}^{k,n}_s
    \end{align*}
whereby $(\tilde{u}^{k,n}) \rightarrow \tilde{u}^k$ in the weak* topology of $L^2\left(\tilde{\Omega}^k; L^\infty\left([0,T];L^2_{\sigma}\right) \right)$ and the weak topology of $L^2\left(\tilde{\Omega}^k; L^2\left([0,T];W^{1,2}_{\sigma}\right) \right)$. Therefore,
\begin{equation} \nonumber \tilde{\mathbbm{E}}^k\left(\norm{\tilde{u}^k - w}^2_{L^\infty\left([0,T];L^2_{\sigma} \right)}\right) \leq \liminf_{n \rightarrow \infty} \tilde{\mathbbm{E}}^k\left(\norm{\tilde{u}^{k,n} - w}^2_{L^\infty\left([0,T];L^2_{\sigma} \right)}\right) 
\end{equation}
and we can see that the computation boils down to treating $\norm{\tilde{u}^{k,n}_t - w_t}^2$. For simplicity in the below, we relabel $\tilde{u}^{k,n}$ by $u$ and remove all $k,n$ superscripts. Let us define
$$v \coloneqq u - w - \mathscr{B}$$
where $\mathscr{B}$ is the boundary corrector from Lemma \ref{masmoudi corrector}. Estimates on $v$ will prove sufficient as $\mathscr{B}$ is of vanishing $L^2$ norm. Observe that $v$ satisfies
\begin{align*}
        v_t = & \,\mathcal{P}_nu_0 - \int_0^t\mathcal{P}_nB(u_s, u_s)\ ds  - \nu_{h}\int_0^t \mathcal{P}_nA_{h} u_s\, ds - \nu_z\int_0^t \mathcal{P}_nA_z u_s\, ds +\frac{1}{2}\int_0^t \sum_{i=1}^\infty\mathcal{P}_n\mathcal{P}\tilde{\mathcal{G}}_i^2u_sds \\ &-  \int_0^t\mathcal{P}_n\mathcal{P} \tilde{\mathcal{G}} u_s d\tilde{\mathcal{W}}_s - u_0 + \int_0^tB(w_s, w_s)\ ds - \mathscr{B}_0 - \int_0^t \partial_t\mathscr{B}_sds
    \end{align*}
which yields the energy identity
\begin{align} \nonumber
        \norm{v_t}^2 = & \,\norm{\mathcal{P}_nu_0 -  u_0  - \mathscr{B}_0}^2 - 2\int_0^t\inner{\mathcal{P}_nB(u_s, u_s)}{v_s} ds  - 2\nu_{h}\int_0^t \inner{\mathcal{P}_nA_{h} u_s}{v_s} ds\\ \nonumber &- 2\nu_z\int_0^t \inner{\mathcal{P}_nA_z u_s}{v_s} ds + 2\int_0^t\inner{B(w_s, w_s)}{v_s} ds  - 2\int_0^t \inner{\partial_t\mathscr{B}_s}{v_s}ds\\ &+\int_0^t \sum_{i=1}^\infty\inner{\mathcal{P}_n\mathcal{P}\tilde{\mathcal{G}}_i^2u_s}{v_s}ds + \int_0^t\sum_{i=1}^\infty \norm{\mathcal{P}_n\mathcal{P}\tilde{\mathcal{G}}_iu_s}^2ds -  2\int_0^t\inner{\mathcal{P}_n\mathcal{P} \tilde{\mathcal{G}} u_s}{v_s} d\tilde{\mathcal{W}}_s.  \label{core energy identity}
    \end{align}
To keep track of all terms, let us write
\begin{equation} \nonumber 
    \norm{v_t}^2 = \sum_{l=0}^8 J_l
\end{equation}
which is the sum of the initial condition plus the eight integrals.
A series of constants will appear in the proof; for convenience, we summarise them and their dependencies here.
\begin{itemize}
    \item $\delta$ is an arbitrary positive parameter to be chosen at the end of the proof, which will be the first parameter fixed.
    \item $\check{c}$ denotes a general non-negative constant changing from line to line, which may depend on $\delta$ but is independent of $\tilde{c}$, $k$ and $n$ below.
    \item $\tilde{c}$ is an arbitrary positive parameter to be chosen at the end of the proof, after $\delta$. The boundary corrector $\mathscr{B}$ is constructed for the choice $\theta = \tilde{c}\nu_h$.
    \item $c$ denotes a general non-negative constant changing from line to line, which may depend on $\delta$ and $\tilde{c}$ but is independent of $k$ and $n$.
    \item $o_k$ denotes a non-negative constant, which may depend on $\delta$ and $\tilde{c}$ but is independent of $n$, such that $o_k$ approaches zero as $k \rightarrow \infty$. $k$ will be chosen large at the end of the proof, after $\delta$ and $\tilde{c}$.
    \item $c_k$ denotes a general non-negative constant changing from line to line, which may depend on $\delta$, $\tilde{c}$ and $k$ but is independent of $n$. In particular, $c_k$ may explode as $k \rightarrow \infty$.
    \item $o_n$ denotes a non-negative constant, which may depend on $\delta$, $\tilde{c}$ and $k$, such that $o_n$ approaches zero as $n \rightarrow \infty$. $n$ will be chosen large at the end of the proof, after $\delta$, $\tilde{c}$ and $n$.
\end{itemize}

As described, the order of selection is $\delta$, $\tilde{c}$, $k$, $n$. Any of the above constants may depend on $T$, the $H^\gamma$ norm of $u_0$, the $L^\infty\left([0,T];H^{\gamma}\right)$ norm of $w$ or the $W^{2,\infty}$ norms of $(\xi_i)$, which will be clear in context. Throughout the proof we shall make frequent use of the bounds given in Lemmas \ref{masmoudi corrector} and \ref{bounds on A} without explicit reference, where in Lemma \ref{masmoudi corrector} the choice is $\theta = \tilde{c}\nu_h $. Notice that all quantities bounded in Lemma \ref{bounds on A} are controlled by the generic $c$ in this proof. Starting with $J_0$,
\begin{align*}
    J_0 \leq 2\norm{\mathcal{P}_nu_0 -  u_0}^2 + 2\norm{\mathscr{B}_0}^2 \leq o_n + 2\norm{\mathscr{M}}_{L^2_z}^2\norm{\mathscr{A}_0}_{L^2_h}^2 \leq o_n + c\left(\nu_h\nu_z \right)^{\frac{1}{2}} = o_n + o_k.
\end{align*}
Moving on to $J_1$, we first deal with the projection $\mathcal{P}_n$. Indeed,
\begin{align*}
    \inner{\mathcal{P}_nB(u_s, u_s)}{v_s} &= \inner{\left(\mathcal{P}_n - I\right)B(u_s, u_s)}{v_s} + \inner{B(u_s, u_s)}{v_s}\\
    &= \inner{B(u_s, u_s)}{\left(\mathcal{P}_n - I\right)(-w_s - \mathscr{B}_s)} + \inner{B(u_s, u_s)}{v_s}
\end{align*}
having expanded out $v_s = u_s - w_s - \mathscr{B}_s$ and used that $(\mathcal{P}_n - I)u_s = 0$. In the first term use that
\begin{align*}
    \abs{\inner{B(u_s, u_s)}{\left(\mathcal{P}_n - I\right)(-w_s - \mathscr{B}_s)}} &\leq \norm{B(u_s, u_s)}_{L^{\frac{6}{5}}}\norm{\left(\mathcal{P}_n - I\right)(w_s + \mathscr{B}_s)}_{L^6}\\ &\leq c\norm{u_s}_{L^3}\norm{u_s}_{H^1}\norm{\left(\mathcal{P}_n - I\right)(w_s + \mathscr{B}_s)}_{H^1}\\ &\leq c\norm{u_s}^{\frac{1}{2}}\norm{u_s}_{H^1}^{\frac{3}{2}}\norm{\left(\mathcal{P}_n - I\right)(w_s + \mathscr{B}_s)}_{H^1}
\end{align*}
having applied the Gagliardo-Nirenberg Inequality $\norm{u_s}_{L^3} \leq c\norm{u_s}^{\frac{1}{2}}\norm{u_s}^{\frac{1}{2}}_{H^1}$. We will return to this term later, and instead move on to $\inner{B(u_s, u_s)}{v_s}$. The approach is to split $u$ into its representation $v - w - \mathscr{B}$, yielding
\begin{align*}
    \inner{B(u_s, u_s)}{v_s} &= \inner{B(u_s, v_s)}{v_s} - \inner{B(u_s, w_s)}{v_s} - \inner{B(u_s, \mathscr{B}_s)}{v_s}\\
    &= - \inner{B(v_s, w_s)}{v_s} + \inner{B(w_s, w_s)}{v_s} + \inner{B(\mathscr{B}_s, w_s)}{v_s}\\ &\qquad \qquad - \inner{B(v_s, \mathscr{B}_s)}{v_s} + \inner{B(w_s, \mathscr{B}_s)}{v_s} + \inner{B(\mathscr{B}_s, \mathscr{B}_s)}{v_s}
\end{align*}
where we have used that $\inner{B(u_s, v_s)}{v_s} = 0$. As these terms come with a minus, note that $\inner{B(w_s, w_s)}{v_s}$ cancels with $J_4$. Therefore, we have established that
\begin{align*}
    &J_1 + J_4 \leq\\ & c\int_0^t\norm{u_s}^{\frac{1}{2}}\norm{u_s}_{H^1}^{\frac{3}{2}}\norm{\left(\mathcal{P}_n - I\right)(w_s + \mathscr{B}_s)}_{H^1}ds\\ & -2 \int_0^t\left(\inner{B(v_s, w_s)}{v_s} + \inner{B(\mathscr{B}_s, w_s)}{v_s} - \inner{B(v_s, \mathscr{B}_s)}{v_s} + \inner{B(w_s, \mathscr{B}_s)}{v_s} + \inner{B(\mathscr{B}_s, \mathscr{B}_s)}{v_s} \right)ds
\end{align*}
and we write the right hand side as $\sum_{l=0}^5J_{1,4}^l$, to be addressed individually. To begin,
$$J^1_{1,4} \leq c\int_0^t\norm{v_s}^2\norm{w_s}_{W^{1,\infty}}ds \leq c\int_0^t\norm{v_s}^2ds $$
and
\begin{align*}
    J^2_{1,4} &\leq c\int_0^t\norm{\mathscr{B}_s}\norm{w_s}_{W^{1,\infty}}\norm{v_s}ds \leq c\int_0^t\norm{\mathscr{M}}_{L^2_z}\norm{\mathscr{A}_s}_{L^2_h}\norm{v_s}ds\\ &\leq  c\int_0^t\left(\tilde{c}\nu_h\nu_z \right)^{\frac{1}{4}}\norm{v_s}ds \leq c\left(\nu_h\nu_z \right)^{\frac{1}{2}} + c\int_0^t\norm{v_s}^2ds = o_k + c\int_0^t\norm{v_s}^2ds
\end{align*}
having used Young's Inequality. Moving on to $J^3_{1,4}$, we decompose
\begin{align*}
   \inner{B(v_s, \mathscr{B}_s)}{v_s}  &= \sum_{j=1}^2\inner{v^j_s\partial_j\mathscr{B}_s}{v_s} + \inner{v^3_s \partial_3\mathscr{B}_s}{v_s}\\
   &= \sum_{l=1}^2\sum_{j=1}^2\inner{v^j_s\partial_j\mathscr{B}^l_s}{v^l_s} + \sum_{j=1}^2\inner{v^j_s\partial_j\mathscr{B}^3_s}{v^3_s} + \sum_{l=1}^2\inner{v^3_s \partial_3\mathscr{B}^l_s}{v^l_s} + \inner{v^3_s \partial_3\mathscr{B}^3_s}{v^3_s}
\end{align*} 
In the first term, we have that
\begin{align*}
\abs{\sum_{l=1}^2\sum_{j=1}^2\inner{v^j_s\partial_j\mathscr{B}^l_s}{v^l_s}} \leq \sum_{l=1}^2\sum_{j=1}^2\norm{\partial_j\mathscr{B}^l_s}_{L^\infty}\norm{v_s}^2 \leq \sum_{l=1}^2\sum_{j=1}^2\norm{\mathscr{M}^h}_{L^\infty_z}\norm{\partial_j\mathscr{A}^l_s}_{L^\infty_h}\norm{v_s}^2 \leq c\norm{v_s}^2.
\end{align*}
whilst in the second term,
\begin{align*}
    \abs{\sum_{j=1}^2\inner{v^j_s\partial_j\mathscr{B}^3_s}{v^3_s}} &= \abs{\sum_{j=1}^2\inner{\mathscr{B}^3_s}{\partial_j\left(v^j_sv^3_s\right)}} \leq \sum_{j=1}^2\left(\abs{\inner{\mathscr{B}^3_s}{\partial_jv^j_sv^3_s}} + \abs{\inner{\mathscr{B}^3_s}{v^j_s\partial_jv^3_s}} \right)\\
    &\leq c\sum_{j=1}^2\norm{\mathscr{B}^3_s}_{L^\infty}\norm{\partial_jv_s}\norm{v_s} \leq c\sum_{j=1}^2\norm{\mathscr{M}^z}_{L^\infty_z}\norm{\mathscr{A}^3_s}_{L^\infty_h}\norm{\partial_jv_s}\norm{v_s}\\
    &\leq c\sum_{j=1}^2 \left(\tilde{c}\nu_h\nu_z\right)^{\frac{1}{2}}\norm{\partial_jv_s}\norm{v_s} \leq \delta\nu_h\nu_z\sum_{j=1}^2 \norm{\partial_jv_s}^2 + c\norm{v_s}^2\\
    &\leq \delta\nu_h\sum_{j=1}^2 \norm{\partial_jv_s}^2 + c\norm{v_s}^2
\end{align*}
observing in the last line that $\nu_z \leq 1$. The third term is the most delicate, for which we will look to invoke Lemma \ref{Hardy Littlewood}. We split the integral into the regions where $\mathscr{B}$ is supported,
$$\sum_{l=1}^2\inner{v^3_s \partial_3\mathscr{B}^l_s}{v^l_s} = \sum_{l=1}^2\left(\inner{v^3_s \partial_3\mathscr{B}^l_s}{v^l_s}_{L^2_{h \times \left(0,\frac{1}{4}\right)}} +  \inner{v^3_s \partial_3\mathscr{B}^l_s}{v^l_s}_{L^2_{h \times \left(\frac{3}{4},1\right)}}\right)$$
and observe that 
\begin{align*}
    &\sum_{l=1}^2\left(\abs{\inner{v^3_s \partial_3\mathscr{B}^l_s}{v^l_s}_{L^2_{h \times \left(0,\frac{1}{4}\right)}}} +  \abs{\inner{v^3_s \partial_3\mathscr{B}^l_s}{v^l_s}_{L^2_{h \times \left(\frac{3}{4},1\right)}}}\right)\\ &\qquad\leq \sum_{l=1}^2\left(\norm{z^2\partial_3\mathscr{B}^l_s}_{L^\infty_{h \times \left(0,\frac{1}{4}\right)}} \norm{\frac{v^3_s}{z}}\norm{\frac{v_s}{z}} + \norm{(1-z)^2\partial_3\mathscr{B}^l_s}_{L^\infty_{h \times \left(\frac{3}{4},1\right)}} \norm{\frac{v^3_s}{1-z}}\norm{\frac{v_s}{1-z}}     \right)\\
    &\qquad\leq \check{c}\sum_{l=1}^2\norm{\partial_3\mathscr{A}^l_s}_{L^\infty_h}\norm{\partial_3v^3_s}\norm{\partial_3v_s}\left(\norm{z^2\partial_3\mathscr{M}^h}_{L^\infty_{\left(0,\frac{1}{4}\right)}} + \norm{(1-z)^2\partial_3\mathscr{M}^h}_{L^\infty_{\left(\frac{3}{4},1\right)}}     \right)\\
    &\qquad\leq \check{c}\left(\tilde{c}\nu_h\nu_z \right)^{\frac{1}{2}}\norm{\partial_3v^3_s}\norm{\partial_3v_s}\\
    &\qquad\leq \check{c}\left(\tilde{c}\nu_h\nu_z \right)^{\frac{1}{2}}\sum_{j=1}^2\norm{\partial_jv_s}\norm{\partial_3v_s}\\
    & \qquad \leq \delta\nu_h\sum_{j=1}^2\norm{\partial_jv_s}^2 + \check{c}\tilde{c}\nu_z \norm{\partial_3v_s}^2
\end{align*}  
having applied Lemma \ref{Hardy Littlewood} as suggested, and crucially using in the penultimate line that $\partial_3v^3 = -\partial_1v^1 - \partial_2v^2$ by the divergence free condition, allowing us to pass to a bound by horizontal derivatives. The remaining term of $J^3_{1,4}$ is simple,
$$ \abs{\inner{v^3_s \partial_3\mathscr{B}^3_s}{v^3_s}} \leq c\norm{\partial_3\mathscr{B}^3_s}_{L^\infty}\norm{v_s}^2 \leq c\norm{\partial_3\mathscr{M}^z}_{L^\infty_z}\norm{\mathscr{A}^3_s}_{L^\infty_h}\norm{v_s}^2 \leq c\norm{v_s}^2$$
producing in total that
$$ J^3_{1,4} \leq 4\delta\nu_h\sum_{j=1}^2\int_0^t\norm{\partial_jv_s}^2ds + 2\check{c}\tilde{c}\nu_z \int_0^t\norm{\partial_3v_s}^2ds + c\int_0^t\norm{v_s}^2ds.$$
Let us move on to $J^4_{1,4}$. We write
$$\inner{B(w_s, \mathscr{B}_s)}{v_s} = \sum_{j=1}^2\inner{w^j_s \partial_j\mathscr{B}_s}{v_s} + \inner{w^3_s \partial_3\mathscr{B}_s}{v_s} $$
and see that
\begin{align*}
    \sum_{j=1}^2\abs{\inner{w^j_s \partial_j\mathscr{B}_s}{v_s}} &\leq \norm{w_s}_{L^\infty}\norm{\partial_j\mathscr{B}_s}\norm{v_s} \leq c\sum_{j=1}^2\norm{\mathscr{M}}_{L^2_z}\norm{\partial_j\mathscr{A}_s }_{L^2_h}\norm{v_s}\\ &\leq c\left(\tilde{c}\nu_h\nu_z\right)^{\frac{1}{4}}\norm{v_s} \leq o_k + c\norm{v_s}^2
\end{align*}
as well as, using Lemma \ref{Hardy Littlewood} and splitting onto the support of $\mathscr{B}$ again,
\begin{align*}
    \abs{\inner{w^3_s \partial_3\mathscr{B}_s}{v_s}} &\leq \abs{\inner{w^3_s \partial_3\mathscr{B}_s}{v_s}_{L^2_{h \times \left(0,\frac{1}{4}\right)}}} + \abs{\inner{w^3_s \partial_3\mathscr{B}_s}{v_s}_{L^2_{h \times \left(\frac{3}{4},1\right)}}}\\
    &\leq \norm{\frac{w^3_s}{z}}_{L^\infty}\norm{z\partial_3\mathscr{B}_s}_{L^2_{h \times \left(0,\frac{1}{4}\right)}}\norm{v_s} + \norm{\frac{w^3_s}{1-z}}_{L^\infty}\norm{(1-z)\partial_3\mathscr{B}_s}_{L^2_{h \times \left(\frac{3}{4},1\right)}}\norm{v_s}\\
    &\leq c\norm{\partial_3w^3_s}_{L^\infty}\norm{\mathscr{A}_s}_{L^2_h}\norm{v_s}\left(\norm{z\partial_3\mathscr{M}}_{L^2_{\left(0,\frac{1}{4}\right)}} + \norm{(1-z)\partial_3\mathscr{M}}_{L^2_{\left(\frac{3}{4},1\right)}} \right)\\
    &\leq c\left(\tilde{c}\nu_h\nu_z\right)^{\frac{1}{4}}\norm{v_s}\\
    &\leq o_k + c\norm{v_s}^2.
\end{align*}
Note that whilst $w$ does not satisfy the Dirichlet boundary condition, it satisfies $w \cdot \underline{n} = 0$ which ensures that $w^3 = 0$ on $\partial \mathscr{O}$, hence the validity of applying Lemma \ref{Hardy Littlewood}. Therefore,
$$ J^4_{1,4} \leq o_k + c\int_0^t\norm{v_s}^2ds.$$
Towards $J^5_{1,4}$, we write
$$\inner{B(\mathscr{B}_s, \mathscr{B}_s)}{v_s} = \sum_{j=1}^2\inner{\mathscr{B}^j_s\partial_j \mathscr{B}_s}{v_s} +\inner{\mathscr{B}^3_s\partial_3 \mathscr{B}_s}{v_s} $$
where
\begin{align*}
    \sum_{j=1}^2\abs{\inner{\mathscr{B}^j_s\partial_j \mathscr{B}_s}{v_s}} &\leq \sum_{j=1}^2\norm{\mathscr{B}^j_s}_{L^\infty}\norm{\partial_j \mathscr{B}_s}\norm{v_s} \leq \sum_{j=1}^2\norm{\mathscr{M}}_{L^\infty_z}\norm{\mathscr{A}^j_s}_{L^\infty_h}\norm{\mathscr{M}}_{L^2_z}\norm{\partial_j\mathscr{A}_s}_{L^2_h}\norm{v_s}\\ &\leq c\left(\tilde{c}\nu_h\nu_z\right)^{\frac{1}{4}}\norm{v_s} \leq o_k + c\norm{v_s}^2
\end{align*}
and
\begin{align*}
    \abs{\inner{\mathscr{B}^3_s\partial_3 \mathscr{B}_s}{v_s}} &\leq \norm{\mathscr{B}^3_s}_{L^\infty}\norm{\partial_3 \mathscr{B}_s}\norm{v_s} \leq \norm{\mathscr{M}^z}_{L^\infty_z}\norm{\mathscr{A}^3_s}_{L^\infty_h}\norm{\partial_3\mathscr{M}}_{L^2_z}\norm{\mathscr{A}_s}_{L^2_h}\norm{v_s}\\
    &\leq c\left(\tilde{c}\nu_h\nu_z\right)^{\frac{1}{2}}\left(\tilde{c}\nu_h\nu_z\right)^{-\frac{1}{4}}\norm{v_s} = c\left(\tilde{c}\nu_h\nu_z\right)^{\frac{1}{4}}\norm{v_s} \leq o_k + c\norm{v_s}^2
\end{align*}
giving that
$$J^5_{1,4} \leq o_k +c\int_0^t\norm{v_s}^2ds. $$
Therefore,
\begin{align*}
    J_1 + J_4 \leq \sum_{l=0}^5J^l_{1,4} &\leq c\int_0^t\norm{u_s}^{\frac{1}{2}}\norm{u_s}_{H^1}^{\frac{3}{2}}\norm{\left(\mathcal{P}_n - I\right)(w_s + \mathscr{B}_s)}_{H^1}ds\\ & \qquad + 4\delta\nu_h\sum_{j=1}^2\int_0^t\norm{\partial_jv_s}^2ds + 2\check{c}\tilde{c}\nu_z \int_0^t\norm{\partial_3v_s}^2ds + o_k + c\int_0^t\norm{v_s}^2ds.
\end{align*}
We proceed with $J_2$, dealing with the projection in a similar way to $J_1$ through
\begin{align*}
    -\inner{\mathcal{P}_nA_hu_s}{v_s} & = -\inner{\left(\mathcal{P}_n - I\right)A_hu_s}{v_s} - \inner{A_hu_s}{v_s}\\
    &\leq \abs{\inner{A_hu_s}{\left(\mathcal{P}_n - I\right)\left(-w_s - \mathscr{B}_s\right)}} - \inner{A_hu_s}{v_s}\\
    &\leq \sum_{j=1}^2\abs{\inner{\partial_ju_s}{\partial_j\left(\mathcal{P}_n - I\right)\left(-w_s - \mathscr{B}_s\right)}} - \inner{A_hu_s}{v_s}\\
    &\leq c\norm{u_s}_{H^1}\norm{\left(\mathcal{P}_n - I\right)\left(w_s + \mathscr{B}_s\right)}_{H^1} - \inner{A_hu_s}{v_s}.
\end{align*}
Again alike $J_1$, we split the second term into
$$-\inner{A_hv_s}{v_s} + \inner{A_hw_s}{v_s} + \inner{A_h\mathscr{B}_s}{v_s}$$
where as usual
$$ -\inner{A_hv_s}{v_s} = -\sum_{j=1}^2\norm{\partial_jv_s}^2.$$
Meanwhile,
\begin{align*}
    \abs{\inner{A_hw_s}{v_s}} \leq \sum_{j=1}^2\abs{\inner{\partial_jw_s}{\partial_jv_s}} \leq \delta \sum_{j=1}^2\norm{\partial_jv_s}^2 + c
\end{align*}
having applied Young's Inequality, and
\begin{align*}
    \abs{\inner{A_h\mathscr{B}_s}{v_s}} &\leq \sum_{j=1}^2\abs{\inner{\partial_j\mathscr{B}_s}{\partial_jv_s}} \leq \sum_{j=1}^2 \norm{\mathscr{M}}_{L^2_z}\norm{\partial_j\mathscr{A}_s}_{L^2_h}\norm{\partial_jv_s}\\ &\leq c\sum_{j=1}^2 \left(\tilde{c}\nu_h\nu_z\right)^{\frac{1}{4}}\norm{\partial_jv_s} \leq \delta \sum_{j=1}^2\norm{\partial_jv_s}^2 + o_k.
\end{align*}
We obtain that
\begin{align*}
    J_2 &\leq c\int_0^t\norm{u_s}_{H^1}\norm{\left(\mathcal{P}_n - I\right)\left(w_s + \mathscr{B}_s\right)}_{H^1}ds - 2\nu_h\left(1-2\delta\right)\sum_{j=1}^2\int_0^t\norm{\partial_jv_s}^2ds + o_k.
\end{align*}
Moving on to $J_3$, by the same argument
$$-\inner{\mathcal{P}_nA_zu_s}{v_s} \leq c\norm{u_s}_{H^1}\norm{\left(\mathcal{P}_n - I\right)\left(w_s + \mathscr{B}_s\right)}_{H^1} - \inner{A_zu_s}{v_s} $$
relying on the fact that $w-\mathscr{B}$ is of zero-trace to conduct the integration by parts. The second term is again split into
$$-\inner{A_zv_s}{v_s} + \inner{A_zw_s}{v_s} + \inner{A_z\mathscr{B}_s}{v_s}$$
where
$$ -\inner{A_zv_s}{v_s} = -\norm{\partial_3v_s}^2.$$
In addition,
\begin{align*}
    \abs{\inner{A_zw_s}{v_s}} = \abs{\inner{\partial_3w_s}{\partial_3v_s}} \leq \delta\norm{\partial_3v_s}^2 + c
\end{align*}
whilst
\begin{align*}
    \abs{\inner{A_z\mathscr{B}_s}{v_s}} &= \abs{\inner{\partial_3\mathscr{B}_s}{\partial_3v_s}} \leq \norm{\partial_3\mathscr{M}}_{L^2_z}\norm{\mathscr{A}_s}_{L^2_h}\norm{\partial_3v_s}\\ &\leq c\left(\tilde{c}\nu_h\nu_z \right)^{-\frac{1}{4}}\norm{\partial_3v_s}
    \leq \delta\norm{\partial_3v_s}^2 + c\left(\nu_h\nu_z \right)^{-\frac{1}{2}}.
\end{align*}
We obtain that
\begin{align*}
    J_3 &\leq c\int_0^t\norm{u_s}_{H^1}\norm{\left(\mathcal{P}_n - I\right)\left(w_s + \mathscr{B}_s\right)}_{H^1}ds - 2\nu_z\left(1-2\delta\right)\sum_{j=1}^2\int_0^t\norm{\partial_jv_s}^2ds + o_k
\end{align*}
due to the fact that
$$\nu_z \left(\nu_h\nu_z \right)^{-\frac{1}{2}} = \left(\frac{\nu_z}{\nu_h} \right)^{\frac{1}{2}} = o_k.$$
Having addressed $J_4$ we turn to $J_5$, which we estimate by
\begin{align*}
    J_5 &\leq 2\int_0^t\abs{\inner{\partial_t\mathscr{B}_s}{v_s}}ds \leq c\int_0^t\norm{\mathscr{M}}_{L^2_z}\norm{\partial_t\mathscr{A}_s}_{L^2_h}\norm{v_s}ds\\
    &\leq c\int_0^t\left(\tilde{c}\nu_h\nu_z\right)^{\frac{1}{4}}\norm{v_s}ds  \leq o_k + c\int_0^t\norm{v_s}^2ds.
\end{align*}
It is time to control the contributions from $\tilde{\mathcal{G}}$, starting with $J_6$ which begins in a familiar way with 
\begin{align*}
    \inner{\mathcal{P}_n\mathcal{P}\tilde{\mathcal{G}}_i^2u_s}{v_s} & = \inner{\left(\mathcal{P}_n - I\right)\mathcal{P}\tilde{\mathcal{G}}_i^2u_s}{v_s} + \inner{\mathcal{P}\tilde{\mathcal{G}}_i^2u_s}{v_s}\\
    &\leq \abs{\inner{\mathcal{P}\tilde{\mathcal{G}}_i^2u_s}{\left(\mathcal{P}_n - I\right)\left(-w_s - \mathscr{B}_s\right)}} + \inner{\mathcal{P}\tilde{\mathcal{G}}_i^2u_s}{v_s}\\
    &\leq \abs{\inner{\tilde{\mathcal{G}}_iu_s}{\tilde{\mathcal{G}}_i^*\left(\mathcal{P}_n - I\right)\left(-w_s - \mathscr{B}_s\right)}} + \inner{\mathcal{P}\tilde{\mathcal{G}}_i^2u_s}{v_s}\\
    &\leq c\norm{\tilde{\xi_i}}_{W^{2,\infty}}^2\norm{u_s}_{H^1}\norm{\left(\mathcal{P}_n - I\right)\left(w_s + \mathscr{B}_s\right)}_{H^1} + \inner{\tilde{\mathcal{G}}_i^2u_s}{v_s}\\
    &\leq c\norm{\xi_i}_{W^{2,\infty}}^2\norm{u_s}_{H^1}\norm{\left(\mathcal{P}_n - I\right)\left(w_s + \mathscr{B}_s\right)}_{H^1} + \inner{\tilde{\mathcal{G}}_i^2u_s}{v_s}
\end{align*}
where we have used the coarse bound that $\nu_h, \nu_z < 1$. In fact to treat the second term we will combine $J_6$ with $J_7$, first using the simple bound
$$\norm{\mathcal{P}_n\mathcal{P}\tilde{\mathcal{G}}_iu_s}^2 \leq \norm{\tilde{\mathcal{G}}_iu_s}^2 $$
by the orthogonal projections. Once more we substitute in $u = v - w - \mathscr{B}$, obtaining that
\begin{align*}
 &\inner{\tilde{\mathcal{G}}_i^2u_s}{v_s} + \norm{\tilde{\mathcal{G}}_iu_s}^2\\ &  = \inner{\tilde{\mathcal{G}}_i^2\left(v_s - w_s - \mathscr{B}_s \right)}{v_s}   + \inner{\tilde{\mathcal{G}}_i(v_s - w_s - \mathscr{B}_s )}{\tilde{\mathcal{G}}_iu_s}\\
 & = \inner{\tilde{\mathcal{G}}_i^2 v_s }{v_s} - \inner{\tilde{\mathcal{G}}_i^2w_s }{v_s} - \inner{\tilde{\mathcal{G}}_i^2\mathscr{B}_s}{v_s} + \inner{\tilde{\mathcal{G}}_iv_s }{\tilde{\mathcal{G}}_iu_s} - \inner{\tilde{\mathcal{G}}_iw_s }{\tilde{\mathcal{G}}_iu_s} - \inner{\tilde{\mathcal{G}}_i\mathscr{B}_s }{\tilde{\mathcal{G}}_iu_s}\\
 & = \inner{\tilde{\mathcal{G}}_i^2 v_s }{v_s} - \inner{\tilde{\mathcal{G}}_i^2w_s }{v_s} - \inner{\tilde{\mathcal{G}}_i^2\mathscr{B}_s}{v_s} + \inner{\tilde{\mathcal{G}}_iv_s }{\tilde{\mathcal{G}}_iv_s} - \inner{\tilde{\mathcal{G}}_iv_s }{\tilde{\mathcal{G}}_iw_s} - \inner{\tilde{\mathcal{G}}_iv_s }{\tilde{\mathcal{G}}_i\mathscr{B}_s}\\
 & \, - \inner{\tilde{\mathcal{G}}_iw_s }{\tilde{\mathcal{G}}_iv_s} + \inner{\tilde{\mathcal{G}}_iw_s }{\tilde{\mathcal{G}}_iw_s} + \inner{\tilde{\mathcal{G}}_iw_s }{\tilde{\mathcal{G}}_i\mathscr{B}_s} - \inner{\tilde{\mathcal{G}}_i\mathscr{B}_s }{\tilde{\mathcal{G}}_iv_s} + \inner{\tilde{\mathcal{G}}_i\mathscr{B}_s }{\tilde{\mathcal{G}}_iw_s} + \inner{\tilde{\mathcal{G}}_i\mathscr{B}_s }{\tilde{\mathcal{G}}_i\mathscr{B}_s}\\
 &= \left(\inner{\tilde{\mathcal{G}}_i^2 v_s }{v_s} + \norm{\tilde{\mathcal{G}}_iv_s }^2\right) - \inner{\tilde{\mathcal{G}}_i^2w_s }{v_s} - \inner{\tilde{\mathcal{G}}_i^2\mathscr{B}_s}{v_s} - 2\inner{\tilde{\mathcal{G}}_iv_s }{\tilde{\mathcal{G}}_iw_s} - 2\inner{\tilde{\mathcal{G}}_iv_s }{\tilde{\mathcal{G}}_i\mathscr{B}_s}\\
 &\qquad +2\inner{\tilde{\mathcal{G}}_iw_s }{\tilde{\mathcal{G}}_i\mathscr{B}_s} + \norm{\tilde{\mathcal{G}}_iw_s }^2 + \norm{\tilde{\mathcal{G}}_i\mathscr{B}_s }^2.
\end{align*}
As before, we denote the corresponding infinite sum and integrals of these terms by $\sum_{l=1}^8J_{6,7}^l$. In $J^1_{6,7}$ we apply (\ref{bigbound1NEW}) to obtain that
$$J^1_{6,7} \leq c\int_0^t\norm{v_s}^2ds + \delta \nu_h\sum_{j=1}^2\int_0^t\norm{\partial_jv_s}^2ds + \delta \nu_z\int_0^t\norm{\partial_3v_s}^2ds.$$
In $J^2_{6,7}$ we write
\begin{align*}
    \inner{\tilde{\mathcal{G}}_i^2w_s }{v_s} &= \inner{\tilde{\mathcal{G}}_iw_s }{\tilde{\mathcal{G}}_i^*v_s}\\ &= \inner{\left(\tilde{\mathcal{G}}_i^h + \tilde{\mathcal{G}}_i^z\right)w_s }{\left(\tilde{\mathcal{G}}_i^{h,*} + \tilde{\mathcal{G}}_i^{z,*}\right) v_s}\\ &= \inner{\tilde{\mathcal{G}}_i^h w_s }{\tilde{\mathcal{G}}_i^{h,*}v_s} + \inner{\tilde{\mathcal{G}}_i^hw_s }{ \tilde{\mathcal{G}}_i^{z,*} v_s} + \inner{\tilde{\mathcal{G}}_i^zw_s }{\tilde{\mathcal{G}}_i^{h,*} v_s} + \inner{\tilde{\mathcal{G}}_i^zw_s }{\tilde{\mathcal{G}}_i^{z,*} v_s}.
\end{align*}
In the following we shall make frequent use of Lemma \ref{bounds on split noise} without explicit reference on each occasion. Firstly,
\begin{align*}
 \sum_{i=1}^\infty \abs{\inner{\tilde{\mathcal{G}}_i^h w_s }{\tilde{\mathcal{G}}_i^{h,*}v_s}} &\leq c\sum_{i=1}^\infty \norm{\tilde{\xi}^h_i}_{W^{1,\infty}}^2\left(\sum_{j=1}^2\norm{\partial_jw_s} + \norm{w_s} \right)\left(\sum_{j=1}^2\norm{\partial_jv_s} + \norm{v_s} \right)\\
 &\leq c\nu_h\left(\sum_{j=1}^2\norm{\partial_jv_s} + \norm{v_s} \right)\\
 &\leq \nu_h\left(\delta \sum_{j=1}^2\norm{\partial_jv_s}^2 + c\norm{v_s}^2 + c \right)\\
 &\leq \delta \nu_h\sum_{j=1}^2\norm{\partial_jv_s}^2 + c\norm{v_s}^2 + o_k.
\end{align*}
In addition,
\begin{align*}
    \sum_{i=1}^\infty \abs{\inner{\tilde{\mathcal{G}}_i^hw_s }{ \tilde{\mathcal{G}}_i^{z,*} v_s}} &\leq c\sum_{i=1}^\infty\norm{\tilde{\xi}^h_i}_{W^{1,\infty}}\norm{\tilde{\xi}^z_i}_{W^{1,\infty}} \left(\sum_{j=1}^2\norm{\partial_jw_s} + \norm{w_s} \right)\left(\norm{\partial_3v_s} + \norm{v_s} \right)\\
    &\leq c(\nu_h\nu_z)^{\frac{1}{2}}\left(\norm{\partial_3v_s} + \norm{v_s} \right)\\
    &\leq \nu_h^{\frac{1}{2}}\left(\delta\nu_z\norm{\partial_3v_s}^2 + c\norm{v_s}^2 + c\right)\\
    &\leq \delta \nu_z\norm{\partial_3v_s}^2 + c\norm{v_s}^2 + o_k
\end{align*}
where on this occasion we applied Young's Inequality by grouping $\nu_z^{\frac{1}{2}}\norm{\partial_3v_s}$ together. An identical computation provides
$$\sum_{i=1}^\infty \abs{\inner{\tilde{\mathcal{G}}_i^zw_s }{ \tilde{\mathcal{G}}_i^{h,*} v_s}} \leq \delta \nu_h\sum_{j=1}^2\norm{\partial_jv_s}^2 + c\norm{v_s}^2 + o_k$$
and following the same steps as the previous term,
$$\sum_{i=1}^\infty \abs{\inner{\tilde{\mathcal{G}}_i^zw_s }{ \tilde{\mathcal{G}}_i^{z,*} v_s}} \leq \delta \nu_z\norm{\partial_3v_s}^2 + c\norm{v_s}^2 + o_k.$$
Altogether,
$$J^2_{6,7} \leq 2\delta\nu_h\sum_{j=1}^2\int_0^t\norm{\partial_jv_s}^2ds + 2\delta\nu_z\int_0^t\norm{\partial_3v_s}^2ds + c\int_0^t\norm{v_s}^2ds + o_k.$$
In $J^3_{6,7}$ we use the same decomposition
\begin{align*}
    \inner{\tilde{\mathcal{G}}_i^2\mathscr{B}_s }{v_s} = \inner{\tilde{\mathcal{G}}_i^h \mathscr{B}_s }{\tilde{\mathcal{G}}_i^{h,*}v_s} + \inner{\tilde{\mathcal{G}}_i^h\mathscr{B}_s }{ \tilde{\mathcal{G}}_i^{z,*} v_s} + \inner{\tilde{\mathcal{G}}_i^z\mathscr{B}_s }{\tilde{\mathcal{G}}_i^{h,*} v_s} + \inner{\tilde{\mathcal{G}}_i^z\mathscr{B}_s }{\tilde{\mathcal{G}}_i^{z,*} v_s}.
\end{align*}
For the first term,
\begin{align*}
    \sum_{i=1}^\infty \abs{\inner{\tilde{\mathcal{G}}_i^h \mathscr{B}_s }{\tilde{\mathcal{G}}_i^{h,*}v_s}} &\leq c\sum_{i=1}^\infty \norm{\tilde{\xi}^h_i}_{W^{1,\infty}}^2\left(\sum_{j=1}^2\norm{\partial_j\mathscr{B}_s} + \norm{\mathscr{B}_s} \right)\left(\sum_{j=1}^2\norm{\partial_jv_s} + \norm{v_s} \right)\\
    &\leq c\nu_h\norm{\mathscr{M}_s}_{L^2_z}\left(\sum_{j=1}^2\norm{\partial_j\mathscr{A}_s}_{L^2_h} + \norm{\mathscr{A}_s}_{L^2_h} \right)\left(\sum_{j=1}^2\norm{\partial_jv_s} + \norm{v_s} \right)\\
    &\leq c\nu_h\left(\sum_{j=1}^2\norm{\partial_jv_s} + \norm{v_s} \right)\\
    &\leq \delta \nu_h\sum_{j=1}^2\norm{\partial_jv_s}^2 + c\norm{v_s}^2 + o_k
\end{align*}
ignoring the useful decay on $\mathscr{M}$ as it is not necessary here. Doing the same again, following $J^2_{6,7}$ as well,
$$ \sum_{i=1}^\infty \abs{\inner{\tilde{\mathcal{G}}_i^h \mathscr{B}_s }{\tilde{\mathcal{G}}_i^{z,*}v_s}} \leq \delta \nu_z\norm{\partial_3v_s}^2 + c\norm{v_s}^2 + o_k.$$
We will have to be much more precise when taking $\partial_3$ derivatives of $\mathscr{B}$. Indeed,
\begin{align*}
    \sum_{i=1}^\infty \abs{\inner{\tilde{\mathcal{G}}_i^z \mathscr{B}_s }{\tilde{\mathcal{G}}_i^{h,*}v_s}} &\leq c\sum_{i=1}^\infty\norm{\tilde{\xi}^h_i}_{W^{1,\infty}}\norm{\tilde{\xi}^z_i}_{W^{1,\infty}} \left(\norm{\partial_3\mathscr{B}_s} + \norm{\mathscr{B}_s} \right)\left(\sum_{j=1}^2\norm{\partial_jv_s} + \norm{v_s} \right)\\
    &\leq c(\nu_h\nu_z)^{\frac{1}{2}}\norm{\mathscr{A}_s}_{L^2_h}\left(\norm{\partial_3\mathscr{M}_s}_{L^2_z} + \norm{\mathscr{M}_s}_{L^2_z}\right)\left(\sum_{j=1}^2\norm{\partial_jv_s} + \norm{v_s} \right)\\
    &\leq c(\nu_h\nu_z)^{\frac{1}{2}}\left(\nu_h\nu_z \right)^{-\frac{1}{4}}\left(\sum_{j=1}^2\norm{\partial_jv_s} + \norm{v_s} \right)\\
    &= c\nu_h^{\frac{1}{2}}\left(\frac{\nu_z}{\nu_h} \right)^{\frac{1}{4}}\left(\sum_{j=1}^2\norm{\partial_jv_s} + \norm{v_s} \right)\\
    &\leq \left(\frac{\nu_z}{\nu_h} \right)^{\frac{1}{4}}\left(\delta \nu_h\sum_{j=1}^2\norm{\partial_jv_s}^2 + c\norm{v_s}^2 + c\right)\\ 
    &\leq \delta \nu_h\sum_{j=1}^2\norm{\partial_jv_s}^2 + c\norm{v_s}^2 + o_k.
\end{align*}
In the last term,
\begin{align*}
    \sum_{i=1}^\infty \abs{\inner{\tilde{\mathcal{G}}_i^z \mathscr{B}_s }{\tilde{\mathcal{G}}_i^{z,*}v_s}} &\leq c\sum_{i=1}^\infty\norm{\tilde{\xi}^z_i}_{W^{1,\infty}}^2 \left(\norm{\partial_3\mathscr{B}_s} + \norm{\mathscr{B}_s} \right)\left(\norm{\partial_3v_s} + \norm{v_s} \right)\\
    &\leq c\nu_z \norm{\mathscr{A}_s}_{L^2_h}\left(\norm{\partial_3\mathscr{M}_s}_{L^2_z} + \norm{\mathscr{M}_s}_{L^2_z}\right)\left(\norm{\partial_3v_s} + \norm{v_s} \right)\\
    &\leq c\nu_z\left(\nu_h\nu_z \right)^{-\frac{1}{4}}\left(\norm{\partial_3v_s} + \norm{v_s} \right)\\
    &= c\nu_z^{\frac{1}{2}}\left(\frac{\nu_z}{\nu_h} \right)^{\frac{1}{4}}\left(\norm{\partial_3v_s} + \norm{v_s} \right)\\
    &\leq \left(\frac{\nu_z}{\nu_h} \right)^{\frac{1}{4}}\left(\delta \nu_z\norm{\partial_3v_s}^2 + c\norm{v_s}^2 + c\right)\\ 
    &\leq \delta \nu_z\norm{\partial_3v_s}^2 + c\norm{v_s}^2 + o_k.
\end{align*}
Summing these terms,
$$J^3_{6,7} \leq 2\delta\nu_h\sum_{j=1}^2\int_0^t\norm{\partial_jv_s}^2ds + 2\delta\nu_z\int_0^t\norm{\partial_3v_s}^2ds + c\int_0^t\norm{v_s}^2ds + o_k. $$
Moreover, the integrals $J^4_{6,7}$ and $J^5_{6,7}$ are treated identically to $J^2_{6,7}$ and $J^3_{6,7}$ respectively, simply replacing $\tilde{\mathcal{G}}_i^*$ by $\tilde{\mathcal{G}}_i$ whose decomposition satisfies the same bounds. Therefore,
\begin{align*}
    J^4_{6,7} &\leq 4\delta\nu_h\sum_{j=1}^2\int_0^t\norm{\partial_jv_s}^2ds + 4\delta\nu_z\int_0^t\norm{\partial_3v_s}^2ds + c\int_0^t\norm{v_s}^2ds + o_k,\\
    J^5_{6,7} &\leq 4\delta\nu_h\sum_{j=1}^2\int_0^t\norm{\partial_jv_s}^2ds + 4\delta\nu_z\int_0^t\norm{\partial_3v_s}^2ds + c\int_0^t\norm{v_s}^2ds + o_k.
\end{align*}
Next up is $J^6_{6,7}$, which follows the same decomposition
\begin{align*}
    \inner{\tilde{\mathcal{G}}_iw_s }{\tilde{\mathcal{G}}_i\mathscr{B}_s} = \inner{\tilde{\mathcal{G}}_i^h w_s }{\tilde{\mathcal{G}}_i^{h}\mathscr{B}_s} + \inner{\tilde{\mathcal{G}}_i^hw_s }{ \tilde{\mathcal{G}}_i^{z} \mathscr{B}_s} + \inner{\tilde{\mathcal{G}}_i^zw_s }{\tilde{\mathcal{G}}_i^{h} \mathscr{B}_s} + \inner{\tilde{\mathcal{G}}_i^zw_s }{\tilde{\mathcal{G}}_i^{z} \mathscr{B}_s}.
\end{align*}
Firstly,
\begin{align*}
    \sum_{i=1}^\infty \abs{\inner{\tilde{\mathcal{G}}_i^{h}w_s}{\tilde{\mathcal{G}}_i^h \mathscr{B}_s }} &\leq c\sum_{i=1}^\infty \norm{\tilde{\xi}^h_i}_{W^{1,\infty}}^2\left(\sum_{j=1}^2\norm{\partial_jw_s} + \norm{w_s} \right)\left(\sum_{j=1}^2\norm{\partial_j\mathscr{B}_s} + \norm{\mathscr{B}_s} \right)\\
    &\leq c\nu_h\left(\sum_{j=1}^2\norm{\partial_jv_s} + \norm{v_s} \right)\norm{\mathscr{M}_s}_{L^2_z}\left(\sum_{j=1}^2\norm{\partial_j\mathscr{A}_s}_{L^2_h} + \norm{\mathscr{A}_s}_{L^2_h} \right)\\
    &= o_k.
\end{align*}
Similarly easy is the third term, for which
$$\sum_{i=1}^\infty \abs{\inner{\tilde{\mathcal{G}}_i^{z}w_s}{\tilde{\mathcal{G}}_i^h \mathscr{B}_s }} = o_k.$$
In the second term,
\begin{align*}
    \sum_{i=1}^\infty \abs{\inner{\tilde{\mathcal{G}}_i^{h}w_s}{\tilde{\mathcal{G}}_i^z \mathscr{B}_s }} &\leq c\sum_{i=1}^\infty\norm{\tilde{\xi}^h_i}_{W^{1,\infty}}\norm{\tilde{\xi}^z_i}_{W^{1,\infty}} \left(\sum_{j=1}^2\norm{\partial_jw_s} + \norm{w_s} \right)\left(\norm{\partial_3\mathscr{B}_s} + \norm{\mathscr{B}_s} \right)\\
    &\leq c(\nu_h\nu_z)^{\frac{1}{2}}\norm{\mathscr{A}_s}_{L^2_h}\left(\norm{\partial_3\mathscr{M}_s}_{L^2_z} + \norm{\mathscr{M}_s}_{L^2_z}\right)\\
    &\leq c(\nu_h\nu_z)^{\frac{1}{2}}\left(\nu_h\nu_z \right)^{-\frac{1}{4}}\\
    &= o_k
\end{align*}
whilst in the fourth,
\begin{align*}
    \sum_{i=1}^\infty \abs{\inner{\tilde{\mathcal{G}}_i^{z}w_s}{\tilde{\mathcal{G}}_i^z \mathscr{B}_s }} &\leq c\sum_{i=1}^\infty\norm{\tilde{\xi}^z_i}_{W^{1,\infty}}^2\left(\norm{\partial_3w_s} + \norm{w_s} \right)\left(\norm{\partial_3\mathscr{B}_s} + \norm{\mathscr{B}_s} \right)\\
    &\leq c\nu_z\norm{\mathscr{A}_s}_{L^2_h}\left(\norm{\partial_3\mathscr{M}_s}_{L^2_z} + \norm{\mathscr{M}_s}_{L^2_z}\right)\\
    &\leq c\nu_z\left(\nu_h\nu_z \right)^{-\frac{1}{4}}\\
    &= c\nu_z^{\frac{1}{2}}\left(\frac{\nu_z}{\nu_h} \right)^{\frac{1}{4}}\\
    &= o_k.
\end{align*}
Simply,
$$ J^6_{6,7} \leq o_k.$$
$J^7_{6,7}$ is almost immediate,
$$J^7_{6,7} \leq c\int_0^t\sum_{i=1}^\infty \norm{\tilde{\xi_i}}_{W^{1,\infty}}^2\norm{w_s}_{H^1}^2ds \leq c(\nu_h + \nu_z) = o_k.  $$
The final term, $J^8_{6,7}$, is decomposed as usual into
$$\norm{\tilde{\mathcal{G}}_i\mathscr{B}_s }^2 = \norm{\tilde{\mathcal{G}}_i^h \mathscr{B}_s }^2 + 2\inner{\tilde{\mathcal{G}}_i^h\mathscr{B}_s }{ \tilde{\mathcal{G}}_i^{z} \mathscr{B}_s} + \norm{\tilde{\mathcal{G}}_i^z\mathscr{B}_s }^2.$$
We have that
\begin{align*}
    \sum_{i=1}^\infty \norm{\tilde{\mathcal{G}}_i^h \mathscr{B}_s }^2 &\leq c\sum_{i=1}^\infty \norm{\tilde{\xi}^h_i}_{W^{1,\infty}}^2\left(\sum_{j=1}^2\norm{\partial_j\mathscr{B}_s} + \norm{\mathscr{B}_s} \right)^2\\
    &\leq c\nu_h\norm{\mathscr{M}_s}_{L^2_z}^2\left(\sum_{j=1}^2\norm{\partial_j\mathscr{A}_s}_{L^2_h} + \norm{\mathscr{A}_s}_{L^2_h} \right)^2\\
    &= o_k
\end{align*}
whilst
\begin{align*}
    \sum_{i=1}^\infty \abs{\inner{\tilde{\mathcal{G}}_i^h \mathscr{B}_s }{\tilde{\mathcal{G}}_i^{z}\mathscr{B}_s}} &\leq c\sum_{i=1}^\infty \norm{\tilde{\xi}^h_i}_{W^{1,\infty}}\norm{\tilde{\xi}^z_i}_{W^{1,\infty}}\left(\sum_{j=1}^2\norm{\partial_j\mathscr{B}_s} + \norm{\mathscr{B}_s} \right)\left(\norm{\partial_3\mathscr{B}_s} + \norm{\mathscr{B}_s} \right)\\
    &\leq c(\nu_h\nu_z)^{\frac{1}{2}}\norm{\mathscr{M}_s}_{L^2_z}\left(\norm{\partial_3\mathscr{M}_s}_{L^2_z} + \norm{\mathscr{M}_s}_{L^2_z}\right)\\
    &\leq c(\nu_h\nu_z)^{\frac{1}{2}}(\nu_h\nu_z)^{\frac{1}{4}}(\nu_h\nu_z)^{-\frac{1}{4}}\\
    &= o_k
\end{align*}
and lastly,
\begin{align*}
    \sum_{i=1}^\infty \norm{\tilde{\mathcal{G}}_i^z \mathscr{B}_s }^2 &\leq c\sum_{i=1}^\infty \norm{\tilde{\xi}^z_i}_{W^{1,\infty}}^2\left(\norm{\partial_3\mathscr{B}_s} + \norm{\mathscr{B}_s} \right)^2\\
    &\leq c\nu_z\norm{\mathscr{A}_s}_{L^2_h}^2\left(\norm{\partial_3\mathscr{M}_s}_{L^2_z} + \norm{\mathscr{M}_s}_{L^2_z} \right)^2\\
    &\leq c\nu_z(\nu_h\nu_z)^{-\frac{1}{2}}\\
    &= c\left(\frac{\nu_z}{\nu_h}\right)^{\frac{1}{2}}\\
    &= o_k.
\end{align*}
Therefore,
$$ J^8_{6,7} \leq o_k$$
so summing over all terms,
\begin{align*}
    J_6 + J_7 &\leq c\int_0^t\norm{u_s}_{H^1}\norm{\left(\mathcal{P}_n - I\right)\left(w_s + \mathscr{B}_s\right)}_{H^1}ds + \sum_{l=1}^8J^l_{6,7}\\
&\leq \int_0^t\norm{u_s}_{H^1}\norm{\left(\mathcal{P}_n - I\right)\left(w_s + \mathscr{B}_s\right)}_{H^1}ds\\ &\qquad + 14\delta\nu_h\sum_{j=1}^2\int_0^t\norm{\partial_jv_s}^2ds + 14\delta\nu_z\int_0^t\norm{\partial_3v_s}^2ds + c\int_0^t\norm{v_s}^2ds + o_k.
\end{align*} 
Only $J_8$ remains to be treated, however this will require first taking the supremum and expectation and then applying the Burkholder-Davis-Gundy Inequality. Before doing this, we collect $J_0$ to $J_7$. Indeed,
\begin{align*}
    \sum_{l=0}^7J_l &\leq c\int_0^t\left(1 +\norm{u_s}^{\frac{1}{2}}\right)\left(1 + \norm{u_s}_{H^1}^{\frac{3}{2}}\right)\norm{\left(\mathcal{P}_n - I\right)(w_s + \mathscr{B}_s)}_{H^1}ds\\ &  -2\nu_h(1-11\delta)\sum_{j=1}^2\int_0^t\norm{\partial_jv_s}^2ds - 2\nu_z(1 - 9\delta -\check{c}\tilde{c}) \int_0^t\norm{\partial_3v_s}^2ds + c\int_0^t\norm{v_s}^2ds + o_k + o_n.
\end{align*}
Substituting this back into (\ref{core energy identity}), we obtain that
\begin{align} \nonumber
    \norm{v_t}^2 &+ 2\nu_h(1-11\delta)\sum_{j=1}^2\int_0^t\norm{\partial_jv_s}^2ds + 2\nu_z(1 - 9\delta -\check{c}\tilde{c}) \int_0^t\norm{\partial_3v_s}^2ds\\ \nonumber &\leq c\int_0^t\left(1 +\norm{u_s}^{\frac{1}{2}}\right)\left(1 + \norm{u_s}_{H^1}^{\frac{3}{2}}\right)\norm{\left(\mathcal{P}_n - I\right)(w_s + \mathscr{B}_s)}_{H^1}ds\\ & \qquad + c\int_0^t\norm{v_s}^2ds -2\int_0^t\inner{\mathcal{P}_n\mathcal{P} \tilde{\mathcal{G}} u_s}{v_s} d\tilde{\mathcal{W}}_s+ o_k + o_n. \nonumber
\end{align}
We now choose the parameter $\delta$ followed by $\tilde{c}$, to be
$$\delta \coloneqq \frac{1}{22}, \qquad \tilde{c} = \frac{1}{11\check{c}}$$
giving that $1 - 11\delta = 1 - 9\delta - \check{c}\tilde{c} = \frac{1}{2}$. Parameters now fixed, taking the supremum over $t \in [0,t^*]$ for any $t^*\in (0,T]$ followed by expectation, we arrive at 
\begin{align} \nonumber
    &\tilde{\mathbbm{E}}\left(\sup_{t \in [0,t^*]}\norm{v_t}^2 + \nu_h\sum_{j=1}^2\int_0^{t^*}\norm{\partial_jv_s}^2ds + \nu_z \int_0^{t^*}\norm{\partial_3v_s}^2ds\right)\\ \nonumber & \qquad \leq c\tilde{\mathbbm{E}}\left[\int_0^{t^*}\left(1 +\norm{u_s}^{\frac{1}{2}}\right)\left(1 + \norm{u_s}_{H^1}^{\frac{3}{2}}\right)\norm{\left(\mathcal{P}_n - I\right)(w_s + \mathscr{B}_s)}_{H^1}ds\right]\\ & \qquad \qquad + c\tilde{\mathbbm{E}}\left(\int_0^{t^*}\norm{v_s}^2ds\right) +c \tilde{\mathbbm{E}}\left(\int_0^{t^*} \sum_{i=1}^\infty \inner{\mathcal{P}_n\mathcal{P} \tilde{\mathcal{G}}_i u_s}{v_s}^2 ds\right)^{\frac{1}{2}} + o_k + o_n \label{new core energy ineq}
\end{align}
having immediately applied the Burkholder-Davis-Gundy Inequality. Let us now address the first term on the right hand side. By applying H\"{o}lder's Inequality over the product space, we have that
\begin{align*}
    &\tilde{\mathbbm{E}}\left[\int_0^{t^*}\left(1 +\norm{u_s}^{\frac{1}{2}}\right)\left(1 + \norm{u_s}_{H^1}^{\frac{3}{2}}\right)\norm{\left(\mathcal{P}_n - I\right)(w_s + \mathscr{B}_s)}_{H^1}ds\right]\\ &  \leq c\left(\tilde{\mathbbm{E}}\int_0^{t^*}\left(1 + \norm{u_s}_{H^1}^{2}\right) ds \right)^{\frac{3}{4}}\left(\tilde{\mathbbm{E}}\int_0^{t^*}\left(1 + \norm{u_s}^{2}\right)\norm{\left(\mathcal{P}_n - I\right)(w_s + \mathscr{B}_s)}_{H^1}^4 ds \right)^{\frac{1}{4}}\\
    & \leq c\left(\tilde{\mathbbm{E}}\int_0^{t^*}\left(1 + \norm{u_s}_{H^1}^{2}\right) ds \right)^{\frac{3}{4}}\left(\tilde{\mathbbm{E}}\left[\sup_{s \in [0,t^*]}\left(1 + \norm{u_s}^{2}\right) \right]\right)^{\frac{1}{4}}\left(\int_0^{t^*}\norm{\left(\mathcal{P}_n - I\right)(w_s + \mathscr{B}_s)}_{H^1}^4 ds \right)^{\frac{1}{4}}.
\end{align*}
As is now quite standard, see for example [\cite{goodair2025zero}] Proposition 2.1 for this choice of noise in the isotropic case, we have the bounds
$$ \tilde{\mathbbm{E}}\left(\sup_{s \in [0,T]}\norm{u_s}^{2} + \int_0^{T}\norm{u_r}_{H^1}^2dr \right) \leq c_k$$
which we recall is independent of $n$, but may explode as $k \rightarrow \infty$. By construction of the boundary corrector $\mathscr{B}$ we have that $w_s+\mathscr{B}_s \in W^{1,2}_{\sigma}$, and as $(\mathcal{P}_n)$ are orthogonal projections in this space then for each fixed $k$, $\norm{\left(\mathcal{P}_n - I\right)(w_s + \mathscr{B}_s)}_{H^1}$ is monotone decreasing to zero as $n \rightarrow \infty$. As the integrals are well defined for each $n$, by the Monotone Convergence Theorem we conclude that
$$\left(\int_0^{t^*}\norm{\left(\mathcal{P}_n - I\right)(w_s + \mathscr{B}_s)}_{H^1}^4 ds \right)^{\frac{1}{4}} \leq o_n$$
so we update (\ref{new core energy ineq}) to 
\begin{align} \nonumber
    &\tilde{\mathbbm{E}}\left(\sup_{t \in [0,t^*]}\norm{v_t}^2 + \nu_h\sum_{j=1}^2\int_0^{t^*}\norm{\partial_jv_s}^2ds + \nu_z \int_0^{t^*}\norm{\partial_3v_s}^2ds\right)\\ &\qquad \leq c\tilde{\mathbbm{E}}\left(\int_0^{t^*}\norm{v_s}^2ds\right) +c \tilde{\mathbbm{E}}\left(\int_0^{t^*} \sum_{i=1}^\infty \inner{\mathcal{P}_n\mathcal{P} \tilde{\mathcal{G}}_i u_s}{v_s}^2 ds\right)^{\frac{1}{2}} + o_k + c_ko_n. \label{new core energy ineq 2}
\end{align}
    Note that whilst $o_n$ is allowed to depend on $k$, so $c_ko_n = o_n$, we write $c_ko_n$ simply to stress this dependence on $k$. To treat the remaining noise term, we again observe that
\begin{align*}
    \inner{\mathcal{P}_n\mathcal{P}\tilde{\mathcal{G}}_iu_s}{v_s}^2 & \leq 2\inner{\left(\mathcal{P}_n - I\right)\mathcal{P}\tilde{\mathcal{G}}_iu_s}{v_s}^2 + 2\inner{\mathcal{P}\tilde{\mathcal{G}}_iu_s}{v_s}^2\\
    &\leq c\norm{\xi_i}_{W^{2,\infty}}^2\norm{u_s}^2\norm{\left(\mathcal{P}_n - I\right)\left(w_s + \mathscr{B}_s\right)}_{H^1}^2 + 2\inner{\tilde{\mathcal{G}}_iu_s}{v_s}^2
\end{align*}
so
\begin{align}
   \nonumber &\tilde{\mathbbm{E}}\left(\int_0^{t^*} \sum_{i=1}^\infty \inner{\mathcal{P}_n\mathcal{P} \tilde{\mathcal{G}}_i u_s}{v_s}^2 ds\right)^{\frac{1}{2}}\\ & \qquad \leq c\tilde{\mathbbm{E}}\left(\int_0^{t^*} \norm{u_s}^2\norm{\left(\mathcal{P}_n - I\right)\left(w_s + \mathscr{B}_s\right)}_{H^1}^2 ds\right)^{\frac{1}{2}} + 2\tilde{\mathbbm{E}}\left(\int_0^{t^*} \sum_{i=1}^\infty \inner{\tilde{\mathcal{G}}_iu_s}{v_s}^2 ds\right)^{\frac{1}{2}} \label{indeedee}
\end{align}
and indeed
\begin{align*}
    &\tilde{\mathbbm{E}}\left(\int_0^{t^*} \norm{u_s}^2\norm{\left(\mathcal{P}_n - I\right)\left(w_s + \mathscr{B}_s\right)}_{H^1}^2 ds\right)^{\frac{1}{2}}\\ & \qquad \leq \left[\tilde{\mathbbm{E}}\left(\sup_{s \in [0,t^*]}\norm{u_s}^2\right) \right]^{\frac{1}{2}}\left(\int_0^{t^*} \norm{\left(\mathcal{P}_n - I\right)\left(w_s + \mathscr{B}_s\right)}_{H^1}^2 ds\right)^{\frac{1}{2}}\\ & \qquad \leq c_ko_n
\end{align*}
as seen just prior. For the second term in (\ref{indeedee}), we write
\begin{align*}
    \inner{\tilde{\mathcal{G}}_iu_s}{v_s}^2 \leq 2\inner{\tilde{\mathcal{G}}_iv_s}{v_s}^2  + 2\inner{\tilde{\mathcal{G}}_iw_s}{v_s}^2 + 2\inner{\tilde{\mathcal{G}}_i\mathscr{B}_s}{v_s}^2 
\end{align*}
and denote the corresponding terms as $\sum_{l=1}^3L_l$. In $L_1$ we apply (\ref{bigbound2NEW}) to see that
\begin{align*}
    L_1 \leq c\tilde{\mathbbm{E}}\left(\int_0^{t^*} \norm{v_s}^4 ds\right)^{\frac{1}{2}} &\leq c\tilde{\mathbbm{E}}\left(\sup_{t \in [0,t^*]}\norm{v_t}^2\int_0^{t^*} \norm{v_s}^2 ds\right)^{\frac{1}{2}}\\ &\leq \frac{1}{2}\tilde{\mathbbm{E}}\left( \sup_{t \in [0,t^*]}\norm{v_t}^2\right) + c\tilde{\mathbbm{E}}\left(\int_0^{t^*} \norm{v_s}^2 ds\right)
\end{align*}
where the first term can be absorbed into the first term on the left hand side of (\ref{new core energy ineq 2}). Next,
\begin{align*}
    L_2 &\leq c\tilde{\mathbbm{E}}\left(\int_0^{t^*}\sum_{i=1}^\infty\norm{\tilde{\xi_i}}_{W^{1,\infty}}^2\norm{w_s}_{H^1}^2 \norm{v_s}^2 ds\right)^{\frac{1}{2}} \leq c\tilde{\mathbbm{E}}\left((\nu_h + \nu_z)\int_0^{t^*}\norm{v_s}^2 ds\right)^{\frac{1}{2}}\\ &\leq c(\nu_h + \nu_z) + c\tilde{\mathbbm{E}}\left(\int_0^{t^*}\norm{v_s}^2 ds\right) \leq o_k + c\tilde{\mathbbm{E}}\left(\int_0^{t^*}\norm{v_s}^2 ds\right).
\end{align*}
For $L_3$ we split the term up into
$$\inner{\tilde{\mathcal{G}}_i\mathscr{B}_s}{v_s}^2 \leq 2\inner{\tilde{\mathcal{G}}_i^h\mathscr{B}_s}{v_s}^2 + 2\inner{\tilde{\mathcal{G}}_i^z\mathscr{B}_s}{v_s}^2.$$
Using the bounds on $\tilde{\mathcal{G}}_i^h\mathscr{B}$ and $\tilde{\mathcal{G}}_i^z\mathscr{B}$ that we have seen for example in $J^3_{6,7}$,
\begin{align*}
    \sum_{i=1}^\infty \inner{\tilde{\mathcal{G}}_i^h\mathscr{B}_s}{v_s}^2 \leq c\nu_h\norm{v_s}^2
\end{align*}
whilst
\begin{align*}
    \sum_{i=1}^\infty \inner{\tilde{\mathcal{G}}_i^z\mathscr{B}_s}{v_s}^2 \leq c\nu_z(\nu_h\nu_z)^{-\frac{1}{2}}\norm{v_s}^2 = c\left(\frac{\nu_z}{\nu_h}\right)^{\frac{1}{2}}\norm{v_s}^2.
\end{align*}
Proceeding identically to $L_2$ with these bounds in place, then
$$ L_3 \leq o_k + c\tilde{\mathbbm{E}}\left(\int_0^{t^*}\norm{v_s}^2 ds\right).$$
Substituting everything back into (\ref{new core energy ineq 2}), 
we achieve that
\begin{align} \label{final achievement}
    \tilde{\mathbbm{E}}\left(\sup_{t \in [0,t^*]}\norm{v_t}^2 + \nu_h\sum_{j=1}^2\int_0^{t^*}\norm{\partial_jv_s}^2ds + \nu_z \int_0^{t^*}\norm{\partial_3v_s}^2ds\right) \leq c\tilde{\mathbbm{E}}\left(\int_0^{t^*}\norm{v_s}^2ds\right) + o_k + c_ko_n. 
\end{align}
As we move towards concluding the proof, we revert back to the full notation used at its beginning. Applying Fubini's Theorem on the right hand side of (\ref{final achievement}), by the standard Gr\"{o}nwall Inequality we see that
\begin{equation} \label{end 1}
    \tilde{\mathbbm{E}}^k\left(\sup_{t \in [0,T]}\norm{\tilde{u}^{k,n}_t - w_t - \mathscr{B}^k_t}^2\right) \leq \left(o_k + c_ko_n\right)e^c = o_k + c_ko_n
\end{equation}
which can be inserted back into (\ref{final achievement}) to obtain
\begin{align}
    \tilde{\mathbbm{E}}^k\left( \nu_h^k\sum_{j=1}^2\int_0^{T}\norm{\partial_j\tilde{u}^{k,n}_s - \partial_jw_s - \partial_j\mathscr{B}^k_s}^2ds + \nu_z^k \int_0^{T}\norm{\partial_3\tilde{u}^{k,n}_s - \partial_3w_s - \partial_3\mathscr{B}^k_s}^2ds\right) \leq o_k + c_ko_n. \label{end 2}
\end{align}
Towards a convergence to $w$, note that
\begin{align*}
    \sup_{t \in [0,T]}\norm{\mathscr{B}^k_t} \leq \sup_{t \in [0,T]}\norm{\mathscr{A}_t}_{L^2_h}\norm{\mathscr{M}^k}_{L^2_z} \leq c\left(\nu_h^k\nu_z^k\right)^{\frac{1}{4}} = o_k
\end{align*}
which combines with (\ref{end 1}) to provide
\begin{equation} \label{end 3}
    \tilde{\mathbbm{E}}^k\left(\sup_{t \in [0,T]}\norm{\tilde{u}^{k,n}_t - w_t }^2\right) \leq o_k + c_ko_n.
\end{equation}
To prove the convergence we fix any small $0 < \varepsilon$ and must show that there exists a $K \in \N$ such that for all $K \leq k$, 
\begin{equation}
    \mathbbm{E}\left(\norm{u^{k} - w }_{L^\infty\left([0,T];L^2_{\sigma}\right)}^2\right) < \varepsilon.
\end{equation}
We fix $K$ large enough so that $o_K$ in (\ref{end 3}) is less than $\frac{\varepsilon}{2}$. Then for any fixed $K \leq k$, as described in the beginning of the proof,
\begin{align*}
    \mathbbm{E}\left(\norm{u^{k} - w }_{L^\infty\left([0,T];L^2_{\sigma}\right)}^2\right) &= \tilde{\mathbbm{E}}^k\left(\norm{\tilde{u}^{k} - w }_{L^\infty\left([0,T];L^2_{\sigma}\right)}^2\right)\\ &\leq \liminf_{n \rightarrow \infty} \tilde{\mathbbm{E}}^k\left(\norm{\tilde{u}^{k,n} - w}^2_{L^\infty\left([0,T];L^2_{\sigma} \right)}\right)\\
    &\leq \liminf_{n \rightarrow \infty}\left( \frac{\varepsilon}{2} + c_ko_n\right)\\ &< \varepsilon 
\end{align*}
which concludes the proof of the first assertion. In view of the second statement, keeping (\ref{end 2}) in mind, it is clear that
$$  \nu_h^k\sum_{j=1}^2\int_0^{T}\left(\norm{ \partial_jw_s}^2 + \norm{\partial_j\mathscr{B}^k_s}^2\right)ds + \nu_z^k \int_0^{T}\norm{\partial_3w_s}^2ds \leq o_k$$
whilst we are more detailed in the last term,
$$\nu_z^k \int_0^{T}\norm{\partial_3\mathscr{B}^k_s}^2ds \leq \nu_z^k\norm{\partial_3\mathscr{M}^k}^2_{L^2_z} \int_0^{T}\norm{\mathscr{A}_s}^2_{L^2_h}ds \leq c\nu_z^k\left(\nu_h^k\nu_z^k \right)^{-\frac{1}{2}} = c\left(\frac{\nu_z^k}{\nu_h^k} \right)^{\frac{1}{2}} = o_k.$$
Therefore by (\ref{end 2}),
\begin{align}
    \tilde{\mathbbm{E}}^k\left( \nu_h^k\sum_{j=1}^2\int_0^{T}\norm{\partial_j\tilde{u}^{k,n}_s }^2ds + \nu_z^k \int_0^{T}\norm{\partial_3\tilde{u}^{k,n}_s }^2ds\right) \leq o_k + c_ko_n \nonumber
\end{align}
and the proof concludes through the same final steps as the first assertion above.

\end{proof}

\textbf{Competing Interests:} The author reports that there are no competing interests to declare.\\

\textbf{Data Availability Statement:} There is no associated data.

\addcontentsline{toc}{section}{References} 
\bibliographystyle{newthing}
\bibliography{Biblio}

\end{document}